\documentclass[10pt]{amsart}
\theoremstyle{plain}
\usepackage{amsmath, amsfonts, amsthm, amssymb,amscd,bbold, hyperref}
\usepackage{graphics}
\usepackage{color}
\usepackage{epsfig}
\usepackage{tikz}
\usetikzlibrary{arrows}
\usepackage{todonotes}
\usepackage[all]{xy}
\graphicspath{ {Figures/} }

\title[]{Annular Khovanov-Lee homology, Braids, and Cobordisms}

\author{J. Elisenda Grigsby}
\thanks{JEG was partially supported by NSF CAREER award DMS-1151671.}
\address{Boston College; Department of Mathematics; 522 Maloney Hall; Chestnut Hill, MA 02467}
\email{grigsbyj@bc.edu}

\author{Anthony M. Licata}
\address{Mathematical Sciences Institute; Australian National University; Canberra, Australia}
\email{anthony.licata@anu.edu.au}

\author{Stephan M. Wehrli}
\thanks{SMW was partially supported by NSF grant DMS-1111680.}
\address{Syracuse University; Department of Mathematics; 215 Carnegie; Syracuse, NY 13244}
\email{smwehrli@syr.edu}

\theoremstyle{plain}
\newtheorem{theorem}{Theorem}
\newtheorem{lemma}{Lemma}
\newtheorem{proposition}{Proposition}
\newtheorem{corollary}{Corollary}

\newtheorem{question}{Question}
\newtheorem{claim}{Claim}
\theoremstyle{definition}

\newtheorem{definition}{Definition}
\newtheorem{remark}{Remark}

\newtheorem*{theoremA}{Theorem~\ref{thm:main}}
\newtheorem*{theoremB}{Theorem~\ref{thm:QP}}
\newtheorem*{theoremC}{Theorem~\ref{thm:RV}}
\newtheorem*{theoremD}{Theorem~\ref{thm:d1=w}}
\newtheorem*{corollaryA}{Corollary~\ref{cor:braidedcob}}
\newtheorem*{corollaryB}{Corollary~\ref{cor:bandrank}}

\newcommand{\N}{\ensuremath{\mathbb{N}}}

\newcommand{\R}{\ensuremath{\mathbb{R}}}
\newcommand{\Z}{\ensuremath{\mathbb{Z}}}

\newcommand{\C}{\ensuremath{\mathbb{C}}}
\newcommand{\F}{\ensuremath{\mathbb{F}}}
\newcommand{\bS}{\ensuremath{\mathbb{S}}}
\newcommand{\Id}{\ensuremath{\mbox{\textbb{1}}}}

\newcommand{\OO}{\ensuremath{\mathbb{O}}}
\newcommand{\XX}{\ensuremath{\mathbb{X}}}

\newcommand{\cP}{\ensuremath{\mathcal{P}}}

\newcommand{\cF}{\ensuremath{\mathcal{F}}}

\newcommand{\Kh}{\ensuremath{\mbox{Kh}}}

\newcommand{\Braid}{\ensuremath{\mathfrak B}}
\newcommand{\sltwo}{\ensuremath{\mathfrak{sl}_2}}

\newcommand{\s}{\sigma}

\begin{document}
\bibliographystyle{plain}
\begin{abstract} We prove that the Khovanov-Lee complex of an oriented link, $L$, in a thickened annulus, $A \times I$, has the structure of a $(\Z \oplus \Z)$--filtered complex whose filtered chain homotopy type is an invariant of the isotopy class of $L \subset (A \times I)$. Using ideas of Ozsv{\'a}th-Stipsicz-Szab{\'o} \cite{OSSUpsilon} as reinterpreted by Livingston \cite{LivingstonUpsilon}, we use this structure to define a family of annular Rasmussen invariants that yield information about annular and non-annular cobordisms. Focusing on the special case of annular links obtained as braid closures, we use the behavior of the annular Rasmussen invariants to obtain a {\em necessary} condition for braid quasipositivity and a {\em sufficient} condition for right-veeringness.
\end{abstract}

\maketitle
\bibliographystyle{plain}

\section{Introduction} \label{sec:intro}
In \cite{K}, Khovanov describes how to associate to any diagram of an oriented link $L \subset S^3$ a bigraded chain complex $(\mathcal{C}(L), \partial)$ whose homology is an invariant of $L$. 

In \cite{Lee}, Lee defines a deformation of Khovanov's construction. Explicitly, she constructs an endomorphism, $\Phi$, of the Khovanov complex that anticommutes with Khovanov's differential $\partial$ and satisfies $\Phi^2 = 0$. The total homology of the resulting bicomplex $(\mathcal{C}(L), \partial + \Phi)$ is also a link invariant, but an uninteresting one: it depends only on the number of components of $L$ and their pairwise linking numbers. 

However (quoting \cite{BN_Morr}), this turns out to be very interesting! In \cite{Rasmussen_Slice}, Rasmussen uses the structure of the Lee complex as a $\Z$--filtered complex and its behavior under the Lee chain maps induced by oriented link cobordisms to define a knot invariant, $s(K) \in 2\Z$,\footnote{This was extended to an integer--valued {\em oriented} link invariant in \cite{BW}.} giving a lower bound on the $4$--ball genus of $K$. Indeed, this lower bound is strong enough to yield a combinatorial proof of the topological Milnor conjecture, first proved by Kronheimer-Mrowka \cite{KM_MilnorConjecture}: that the $4$--ball genus of the $(p,q)$ torus knot $T_{p,q}$ is $\frac{(p-1)(q-1)}{2}$, precisely what is predicted by realizing $T_{p,q}$ as the closure of a positive braid.

In a different direction, Asaeda-Przytycki-Sikora \cite{APS} and L. Roberts \cite{LRoberts} define a version of Khovanov homology for links $L$ in a solid torus equipped with an identification as a thickened annulus, $A \times I$.  This variant has come to be known as the (sutured) annular Khovanov homology of $L \subset A \times I$. Imbedding $A \times I \subset S^3$ as the complement of a neighborhood of a standardly-imbedded unknot, one can regard the Khovanov complex as a deformation of the sutured annular Khovanov complex just as the Lee complex is a deformation of the Khovanov complex. Explicitly, one can decompose the Khovanov differential of the complex associated to $L \subset A \times I \subset S^3$ as the sum of two anti-commuting differentials: \[\partial = \partial_0 + \partial_-\] and thus endow the Khovanov complex with the structure of a $\Z$--filtered complex, a structure that has proven to be particularly well-suited for studying braids and their conjugacy classes. It detects non-conjugate braids related by exchange moves \cite{Hubbard}, detects the trivial braid conjugacy class \cite{BaldGrig}, and distinguishes braid closures from other tangle closures \cite{GN}.

The purpose of the present paper is to investigate how the algebraic structure of the {\em Lee} complex of an annular link $L \subset A \times I \subset S^3$ can be used to extract topological information about the link. In particular, the Lee complex of $L \subset A \times I$ is $\Z \oplus \Z$--filtered. Choosing an orientation $o$ on $L$ and using ideas of Ozsv{\'a}th-Stipsicz-Szab{\'o} \cite{OSSUpsilon} as reinterpreted by Livingston \cite{LivingstonUpsilon}, we can then define a family of Rasmussen invariants, $d_t(L,o)$, one for each $t \in [0,2]$, whose value at $t=0$ essentially agrees with the Rasmussen invariant. This family enjoys many of the structural features of the Heegaard-Floer Upsilon invariant, but is invariant only up to isotopy in $(A \times I)$, not $S^3$:

\begin{theoremA} Let $L \subset (A \times I)$ be an annular link, let $o$ be an orientation on $L$, and let $t \in [0,2]$.
\begin{enumerate}
	\item $d_t(L,o)$ is an oriented annular link invariant.
	\item $d_{1-t}(L,o) = d_{1+t}(L,o)$ for all $t \in [0,1]$.
	\item $d_0(L,o) = d_2(L,o) = s(L,o) - 1.$
	\item Viewed as a function $[0,2] \rightarrow \mathbb{R}$, $d_t(L,o)$ is piecewise linear. 
	\item Suppose $(L,o)$ and $(L',o')$ are non-empty oriented annular links, and $F$ is an oriented cobordism from $(L,o)$ to $(L',o')$ for which each component of $F$ has a boundary component in $L$. Then if $F$ has $a_0$ even-index annular critical points, $a_1$ odd-index annular critical points, and $b_0$ even-index non-annular critical points, then \[d_t(L,o) - d_t(L',o') \leq (a_1-a_0) - b_0(1-t).\]
\end{enumerate}
\end{theoremA}

When $(L,o)$ is the {\em annular closure} of a braid $\sigma \in \Braid_n$ equipped with its braid-like orientation $o_{\uparrow}$ (see Subsection \ref{subsec:annbraid}), we can say more. In particular, the work of Hughes in \cite{Hughes} tells us that every link cobordism is isotopic to a so-called {\em braided cobordism} (Definition \ref{defn:braidedcob}), and this isotopy can be realized rel boundary if the original links are already braided with respect to a common axis. This has the following consequence for the annular Rasmussen invariants:

\begin{corollaryA} If $\sigma_0 \in \Braid_{n_0}$ and $\sigma_1 \in \Braid_{n_1}$ are braids, and $F$ is a braid-orientable braided cobordism from $\widehat{\sigma}_0$ to $\widehat{\sigma}_1$ with $a_1$ (annular) odd-index critical points and $b_0$ even-index (non-annular) critical points, and  each component of $F$ has a component in $\sigma_0$, then \[d_t(\widehat{\sigma}_0) - d_t(\widehat{\sigma}_1) \leq a_1 - b_0(1-t).\]
\end{corollaryA}

Noting that a braided cobordism from a braid closure $\widehat{\sigma}$ to the empty link must pass through a braided cobordism from $\widehat{\sigma}$ to the $1$--braid closure $\widehat{\Id}_1$, the above statement remains valid even when $\sigma_1$ is the empty braid (in this case, $d_t(\widehat{\sigma}_1) = 0$).

From the above we also obtain a bound on the so-called {\em band rank} of a braid (conjugacy class) $\sigma \in \Braid_n$ (defined in \cite{Rudolph} and denoted $\mbox{rk}_n(\sigma)$ there). This is the smallest $c \in \Z^{\geq 0}$ for which $\sigma$ can be expressed as a product of $c$ conjugates of elementary Artin generators (either positive or negative). That is, letting $\sigma_k$ denote the $k$th elementary Artin generator, \[\mbox{rk}_n(\sigma) := \mbox{min}\left\{c \in \Z^{\geq 0} \,\,\left|\,\, \sigma = \prod_{j=1}^c \omega_j \sigma_{i_j}^\pm (\omega_j)^{-1} \mbox{ for some $\omega_j, \sigma_{i_j} \in \Braid_n$.}\right.\right\}\] Note that the absolute value of the {\em writhe} of $\sigma$ is a lower bound for $\mbox{rk}_n(\sigma)$, and the length of any word representing $\sigma$ is an upper bound. We obtain:

\begin{corollaryB} Let $\sigma \in \Braid_n$. Then \[\mbox{max}_{t \in [0,2]} \left|d_t(\widehat{\sigma}) - d_t(\widehat{\Id}_n)\right| \leq rk_n(\sigma).\] Noting that $d_t(\widehat{\Id}_n) = -|n(1-t)|$, this bound can be rewritten: \[\mbox{max}_{t \in [0,2]} \left|d_t(\widehat{\sigma}) + |n(1-t)|\right| \leq rk_n(\sigma).\]
\end{corollaryB}
 
We also obtain information about the positivity of $\widehat{\sigma}$, viewed as a mapping class.

Explicitly, let $D_n$ denote the unit disk in $\C$, equipped with $n$ distinct marked points $p_1, \ldots, p_n$. Let $\Delta := \{p_1, \ldots, p_n\}$. For convenience, we will also mark a point $* \in \partial D_n$. A braid $\sigma$ is said to be {\em quasipositive} if it is expressible as a product of conjugates of positive elementary Artin generators (see Definition \ref{defn:QP}) and {\em right-veering} if it sends all admissible arcs from $*$ to $\Delta$ {\em to the right} (see Definition \ref{defn:RV}). Note that all quasipositive braids are right-veering, but many right-veering braids are not quasipositive. The set of non-quasipositive right-veering braids is of significant interest to contact and symplectic geometers, and as yet poorly-understood (cf. \cite{BaldPlam, PlamRV}).

As an application of Theorem \ref{thm:main}, we obtain a {\em necessary} condition for a braid to be quasipositive, and a {\em sufficient} condition for a braid to be right-veering. In particular, we find that $d_t(\widehat{\sigma},o_\uparrow)$ is piecewise linear, with slope bounded above by $n$, the braid index of $\sigma$. Letting $m_t(\widehat{\sigma})$ denote the (right-hand) slope at $t \in [0,2)$ (see Part (4) of the detailed version of Theorem \ref{thm:main} in Section \ref{sec:mainthm}), we find:

\begin{theoremB} If $\sigma \in \Braid_n$ is quasipositive, then $m_t(\widehat{\sigma}) = n$ for all $t \in [0,1)$.
\end{theoremB}

\begin{theoremC} If $m_t(\widehat{\sigma}) = n$ for some $t \in [0,1)$, then $\sigma$ is right-veering.
\end{theoremC}

Our hope is that this will provide a new means of probing and organizing the collection of right-veering non-quasipositive braids. 

In Section \ref{sec:examples}, we provide an example of a braid whose non-quasipositivity and right-veeringness are ensured by its annular Rasmussen invariant. We also describe a number of other examples that give answers to some natural questions one might ask about the effectiveness of the annular Rasmussen invariant at detecting right-veeringness and quasipositivity. 

All of our $d_t$ invariant computations were carried out using Mathematica code written for us by Scott Morrison.  His ideas and input were also instrumental at numerous points in the early stages of this project.  At present we have only used Morrison's program to compute $d_t$ for braids whose length in the standard Artin generators $\sigma_i^{\pm1}$ is at most 11.  As a result, we have only a few examples of braids $\sigma \in \Braid_n$ whose annular Rasmussen invariant has {\em multiple slopes} on the interval $[0,1)$. A partial explanation for lack of multiple slope examples amongst small braids is the following quite strong constraint:

\begin{theoremD} Let $\sigma \in \Braid_n$ have writhe $w$.\footnote{Here (and throughout), we mean the writhe with respect to either braid-like orientation. This is sometimes called the {\em algebraic length} or {\em exponent sum} of the braid.} Then $d_1(\widehat{\sigma}, o_\uparrow) = w$.
\end{theoremD}

One basic feature of the $d_t$ invariant which we do not understand is the behaviour of $d_t$ under positive and negative stabilization.  Note that, by part $(3)$ of Theorem \ref{thm:main}, the values $d_0$ and $d_2$ are invariant under both positive and negative stabilization; by Theorem \ref{thm:d1=w}, on the other hand, $d_1$ increases by $1$ under positive stabilization and decreases by $1$ under negative stabilization.  We discuss this, along with the relationship to the question of the effectiveness of transverse invariants obtained from Khovanov homology (cf. \cite{Plam}) in Section \ref{sec:transverse}.

\subsection*{Acknowledgements} The authors would like to thank John Baldwin, Peter Feller, Jen Hom, Diana Hubbard, Adam Levine, and Olga Plamenevskaya for interesting conversations. We are especially grateful to Scott Morrison for both many helpful conversations and for his generous computational assistance. The first author would also like to thank the BC and Brandeis students in her spring 2016 graduate class for offering useful feedback on preliminary versions of this material.

\section{Algebraic Preliminaries} \label{sec:algprelim}

\begin{definition} Let $I$ be a partially-ordered set. A {\em descending $I$--filtration} on a chain complex $\mathcal{C}$ is the choice of a subcomplex $\cF_i \subseteq \mathcal{C}$ for each $i \in I$,  satisfying the property that if $i \leq i'$ then $\cF_i \supseteq \cF_{i'}$.
\end{definition}

A map $f: \mathcal{C} \rightarrow \mathcal{C}'$ between two complexes with $I$--filtrations $\{\cF_i\}_{i \in I}$ and $\{\cF_i'\}_{i \in I}$ is said to be {\em filtered} if $f(\cF_i) \subseteq \cF'_i$ for all $i \in I$.

\begin{definition} Two $I$--filtered chain complexes $(\mathcal{C}, \partial)$ and $(\mathcal{C}', \partial')$ are said to be $I$--filtered chain homotopy equivalent if there exists a chain homotopy equivalence between $\mathcal{C}$ and $\mathcal{C}'$ for which all involved maps are filtered. Explicitly, there exist filtered chain maps \[f: \mathcal{C} \rightarrow \mathcal{C'} \hskip 10pt \mbox{ and } \hskip 10pt g: \mathcal{C'} \rightarrow \mathcal{C}\] and filtered chain homotopies \[H: \mathcal{C} \rightarrow \mathcal{C} \hskip 10pt  \mbox{ and } \hskip 10pt H':\mathcal{C}' \rightarrow \mathcal{C}\] for which \[gf - \Id_{\mathcal{C}} = H\partial + \partial H\hskip 10pt \mbox{ and } \hskip 10pt fg - \Id_{\mathcal{C}'} = H'\partial' + \partial ' H'.\] 
\end{definition}

The filtered complexes we consider will satisfy some additional desirable properties.

\begin{definition} A descending $I$--filtered complex $\{\cF_i\}_{i \in I}$ is said to be 
\begin{itemize}
	\item {\em discrete} if $\cF_m/\cF_{m'}$ is finite-dimensional for all $m \leq m'$ and
	\item {\em bounded} if there exist some $m, m' \in I$ with $\cF_m = 0$ and $\cF_{m'} = \mathcal{C}$.
\end{itemize}
\end{definition}

In what follows, whenever we mention an $I$--filtered complex, the reader may assume $I$ is either $\R$ or $\Z \oplus \Z$. We shall regard $\Z \oplus \Z$ as a partially-ordered set using the rule $(a,b) \leq (a',b')$ iff $a \leq a'$ and $b \leq b'$.

We can now make the following additional definition:

\begin{definition} A map $f: \mathcal{C} \rightarrow \mathcal{C}'$ between two $I$--filtered complexes is said to be {\em filtered of degree $j \in I$} if $f(\cF_i) \subseteq \cF_{i+j}$ for all $i \in I$.
\end{definition}

Very often, we obtain the structure of a (descending) $I$--filtration when the underlying vector space of $\mathcal{C}$ is $I$--graded, and the differential, $\partial$, is monotonic (non-negative) with respect to the grading. That is, $\partial$ decomposes as $\partial = \sum_{j \geq {\bf 0}} \partial_j,$ where $\partial_j$ is degree $(j \geq {\bf 0}) \in I$ with respect to the $I$--grading on the vector space underlying $\mathcal{C}$. In this case, we will say that the $I$--filtration is {\em induced} by an $I$--grading on $\mathcal{C}$ and we will call any graded basis for $\mathcal{C}$ a {\em filtered graded basis} for the $I$--filtration. All of the filtered complexes considered in the present work will come equipped with a distinguished filtered graded basis.

\begin{remark}An $I$--filtered complex with a finite filtered graded basis is discrete and bounded. 
\end{remark}

\begin{definition} Let $I$ be totally ordered, and suppose $(\mathcal{C}, \partial)$ is a discrete descending $I$--filtered complex with the property that for every $x \neq 0 \in \mathcal{C}$ the set $\{i \in I \,\,|\,\,x \in \cF_i\}$ has a maximal element. Then we will say $\mathcal{C}$ {\em admits a grading.} If $x \neq 0$, we will denote its {\em filtration grading} by \[\mbox{gr}(x) := \mbox{max}\{i \in I\,\,|\,\,x \in \cF_i\}.\]
\end{definition}

\begin{remark} Let $I$ be totally ordered. Not every discrete, bounded $I$--filtered complex admits a grading. For example, let $I = \R$, and consider a $1$--dimensional $\R$--filtered complex with dim($\cF_k$) = 1 for all $k < 0$, but dim($\cF_0$) = 0.
\end{remark}

\begin{remark} If the $I$--filtration on $(\mathcal{C}, \partial)$ is induced by an $I$--grading on $\mathcal{C}$, then $\mathcal{C}$ clearly admits a grading. Moreover, the grading of a homogeneous element of $\mathcal{C}$ coincides with the definition given above.
\end{remark}

Now let us focus on the case where $\mathcal{C}$ is a discrete, bounded $\R$--filtered complex admitting a grading. Then each nonzero homology class of $H_*(\mathcal{C})$ inherits a grading as follows (cf. \cite[Sec. 3]{Rasmussen_Slice}, \cite[Defn. 5.1]{LivingstonUpsilon}).

\begin{definition} Let $\mathcal{C}$ be a finite-dimensional complex endowed with a discrete, bounded, $\R$--valued filtration $\{\cF_s\}_{s \in \R}$ admitting a grading. If $\eta \neq 0 \in H_*(\mathcal{C})$, then \[\mbox{gr}(\eta) := \mbox{max}_{[x] = \eta}\{\mbox{gr(x)} \in \R\}.\]
\end{definition}

\begin{remark}
The fact that $\mathcal{C}$ is finite-dimensional ensures that for every $\eta \neq 0 \in H_*(\mathcal{C})$, the set $\{\mbox{gr(x)} \,\,|\,\,[x] = \eta\}$ is finite, hence has a maximum value. 
\end{remark}

We will have particular interest in families of $\R$--filtrations obtained from a fixed $(\Z \oplus \Z)$--filtered complex equipped with a filtered graded basis. 

Explicitly, let $\mathcal{C}$ be a (descending) $(\Z \oplus \Z)$--filtered complex. Then for each $\theta \in [0,\pi/2]$ we can endow $\mathcal{C}$ with the structure of an $\R$--filtered complex by projecting to the line $\ell_\theta$ making angle $\theta$ with the positive $x$--axis in the plane containing the lattice $\Z \oplus \Z$. Explicitly, for fixed $\theta \in [0, \pi/2]$ and $s \in \R$, define \[\mathcal{F}^\theta(\mathcal{C})_s := \bigcup_{(a,b)\cdot (\cos\theta,\sin\theta) \geq s} \mathcal{F}(\mathcal{C})_{(a,b)}.\]

If $\mathcal{C}$ is equipped with a filtered graded basis $\mathcal{B}$, then this filtered graded basis will descend via the same $\ell_\theta$--projection as above to give a filtered graded basis for each $\R$--filtration $\{\mathcal{F}^\theta(\mathcal{C})_s\}_{s \in \R}$.
 
If $\mathcal{C}$ is discrete (resp., bounded) as a $(\Z \oplus \Z)$--filtered complex, then the resulting $\R$--filtration will be discrete (resp., bounded).

In particular, if $\mathcal{C}$ is a discrete, bounded $(\Z \oplus \Z)$--filtered complex with a finite filtered graded basis, and $\eta \neq 0 \in H_*(\mathcal{C})$, then $\eta$ has a well-defined grading with respect to each $\R$--filtration associated to $\theta$ for $\theta \in [0,\pi/2]$ as above. In the remainder of this section, we give one concrete way to understand these gradings. 

In  Definitions \ref{defn:grsupport} and \ref{defn:subseteta} and Lemma \ref{lem:latticegraddef}, assume $\mathcal{C}$ is a discrete $(\Z \oplus \Z)$--filtered complex with a filtered graded basis $\mathcal{B}$. If $x \in \mathcal{C}$ is $(\Z \oplus \Z)$--homogeneous with grading $(a,b)$, and $\theta \in [0,\pi/2]$, let gr$_\theta(x) = (a,b) \cdot (\cos \theta,\sin \theta)$, as above. 

\begin{definition} \label{defn:grsupport} Let $x \in \mathcal{C}$ be a cycle. Then $x$ is said to be {\em supported on} the subset $S \subseteq (\Z \oplus \Z)$ if $x$ can be decomposed into $(\Z \oplus \Z)$--homogeneous terms, the union of whose gradings is precisely the set $S$. That is, we can express: \[x = \sum_{(a,b) \in S} x_{(a,b)}\]  with gr$(x_{(a,b)}) = (a,b)$, $x_{(a,b)} \neq 0$.
\end{definition}

\begin{definition} \label{defn:subseteta} Let $\eta \neq 0 \in H_*(C)$. Define \[\bS(\eta) := \{S \subseteq \Z \oplus \Z \,\,|\,\,\exists \,\, x \in \mathcal{C} \mbox{ with } [x] = \eta \mbox{ and } x \mbox{ supported on }S.\}\]
\end{definition}

The following lemma is immediate from the definitions.

\begin{lemma} \label{lem:latticegraddef} Let $\eta \neq 0 \in H_*(\mathcal{C})$. Then 
\[\mbox{gr}_\theta(\eta) = \mbox{max}_{S \in \bS(\eta)}\{ \mbox{min}_{(a,b) \in S}\{(a,b) \cdot (\cos\theta, \sin\theta)\}\}\]
\end{lemma}

Informally, we have a number of cycles in $\mathcal{C}$ representing $\eta$. Each such cycle is supported on some subset $S \subseteq \Z \oplus \Z$, and the collection, $\bS(\eta)$, of all such subsets is the information we need to compute gr$_\theta(\eta)$. 

The reader may find the following analogy\footnote{which arose in a conversation with Charlie Frohman} useful.  We can view the process of computing the grading of $\eta$ as a race, judged by $\theta \in [0,\pi/2]$. Each $S \in \bS(\eta)$ is a competing team; its ``time" ($\theta$--grading) is determined by its ``slowest" member (minimal $\theta$--grading among $(a,b) \in S$). The $\theta$--grading of $\eta$ is therefore the $\theta$--grading of the ``slowest" member of the ``fastest" team.
\vskip 10pt
We also note:
\begin{lemma} \label{lem:ZplusZtoR}
If $f: \mathcal{C} \rightarrow \mathcal{C}'$ is a $(\Z \oplus \Z)$--filtered chain homotopy equivalence between discrete, bounded $(\Z \oplus \Z)$--filtered chain complexes $\mathcal{C}, \mathcal{C}'$, then for each $\theta \in \left[0,\pi/2\right]$, \[f: \{\cF^\theta(\mathcal{C})_s\}_{s \in \R} \rightarrow \{\cF^\theta(\mathcal{C}')_s\}_{s \in \R}\] is an $\R$--filtered chain homotopy equivalence. Moreover, if $\mathcal{C}, \mathcal{C}'$ have finite filtered, graded bases then for each $\eta \neq 0 \in H_*(\mathcal{C})$ and each $\theta \in [0, \pi/2]$, we have \[\mbox{gr}_\theta(\eta) = \mbox{gr}_\theta(f_*(\eta)).\]
\end{lemma}

\begin{proof} Nearly immediate from the definitions.
\end{proof}

\section{Topological Preliminaries} \label{sec:topprelim}

Let $A$ be a closed, oriented annulus, $I = [0,1]$ the closed, oriented unit interval. Via the identification 
\[A \times I= \{(r,\theta,z)\,\,\vline\,\,r \in [1,2], \theta \in [0, 2\pi), z\in [0,1]\} \subset (S^3 = \R^3 \cup \infty) ,\] any link, $L \subset A \times I$, may naturally be viewed as a link in the complement of a standardly imbedded unknot, $(U = z$--axis $\cup \,\,\infty) \subset S^3$. Such an {\em annular link} $L \subset A \times I$ admits a diagram, $\cP(L) \subset A,$ obtained by projecting a generic isotopy class representative of $L$ onto $A \times \{1/2\}$. 

For convenience, we shall view $\cP(L) \subset A$ instead as a diagram on $S^2 - \{\OO, \XX\}$, where $\XX$ (resp., $\OO$) are basepoints on $S^2$ corresponding to the inner (resp., outer) boundary circles of $A$. Note that if we forget the data of $\XX$, we may view $\cP(L)$ as a diagram on $\R^2 = S^2 - \{\OO\}$ of $L$, viewed as a link in $S^3$. 

We will focus particular attention in the present work on {\em annular braid closures}. Precisely, let $\sigma \in B_n$ be an $n$--strand braid. Then its {\em annular closure} is the annular link obtained by regarding the classical closure, $\widehat{\sigma}$, of the braid as a link in the complement of the braid axis, $U = z$--axis $\cup \,\, \infty$. That is, $\widehat{\sigma} \subset S^3 - N(U) \sim (A \times I)$.

We will also be interested in {\em oriented link cobordisms} between oriented links $(L,o), (L',o') \subset (A \times I) \subset S^3$: smoothly properly imbedded surfaces in $F \subset S^3 \times I$ with \[\partial F = \left((L,\overline{o}) \subset -S^3 \times \{0\}\right) \amalg \left((L,o) \subset S^3 \times \{1\}\right),\] considered up to isotopy rel boundary. Letting $(U = z$--axis $\cup \,\,\infty) \subset S^3$ as above, we will refer to $F$ as an {\em annular cobordism} (cf. \cite[Appx.]{SchurWeyl}) if  $F \cap (U \times I) = \emptyset$. 

In this case, in any {\em annular movie} (cf. \cite[Appx.]{SchurWeyl}) describing $F$, two {\em non-critical annular stills} separated by a single {\em elementary string interaction} differ by a Reidemeister move, birth, death, or saddle localized away from $\{\OO,\XX\}$. Accordingly, we will say that a planar isotopy (resp., a Reidemeister move, birth, death, or saddle) of an annular diagram $\cP(L)$ of an annular link $L \subseteq A \times I$ is {\em annular} if the local diagram describing the move is supported in a disk contained in $S^2 - \{\OO, \XX\}.$ If the local diagram cannot be made disjoint from $\{\OO,\XX\}$, the isotopy (resp., Reidemeister move, birth, death, or saddle) is said to be {\em non-annular}. Note that by transversality, saddle moves (odd-index critical points of $F$) may always be assumed annular, but planar isotopies, Reidemeister moves, births and deaths (even-index critical points) need not be. In particular, an {\em annular} birth (resp., death) is the addition (resp., deletion) of a {\em trivial} circle, and a {\em non-annular} birth (resp., death) is the addition (resp., deletion) of a {\em non-trivial} circle.

\subsection{Annular Khovanov-Lee complex}
From the data of the diagram $\cP(L) \subset S^2 - \{\OO, \XX\}$ of an oriented annular link $L \subset A \times I$, we will use a construction of Khovanov \cite{K}, along with ideas of Lee \cite{Lee}, Rasmussen \cite{Rasmussen_Slice}, Asaeda-Przytycki-Sikora \cite{APS} and L. Roberts \cite{LRoberts} (see also \cite{AnnularLinks, SchurWeyl}) to define a $(\Z \oplus \Z)$--filtered chain complex as follows.

Begin by temporarily forgetting the data of $\XX$ and construct the standard Lee complex $(\mathcal{C}, \partial^{Lee}),$ (using $\F = \C$ coefficients) associated to $\cP(L)$ as described in \cite{Lee}.

That is, choose an ordering of the crossings of $\cP(L)$ and form the so-called {\em cube of resolutions} of $\cP(L)$ as described in \cite[Sec. 4.2]{K}. This cube of resolutions determines a finite-dimensional bigraded vector space \[\mathcal{C} = \bigoplus_{i,j \in \Z} \mathcal{C}^{i,j}\] along with an endomorphism called the {\em Lee differential}, $\partial^{Lee}: \mathcal{C} \rightarrow \mathcal{C}$, which splits as a sum of two bigrading-homogeneous maps, $\partial$ and $\Phi$. The first of these is the Khovanov differential, and the second is the Lee deformation. Each is degree $1$ with respect to the ``i" (homological) grading, and their ``j" (quantum) degrees are $0$ and $4$, respectively \cite{Lee}. Khovanov \cite[Prop. 8]{K} proves that $\partial^2 = 0$, and Lee \cite[Sec. 4]{Lee} proves that $\Phi$ satisfies:

\begin{itemize}
	\item $\partial\Phi + \Phi\partial = 0$ and
	\item $\Phi^2 = 0$.	
\end{itemize}

The homology of the complex $(\mathcal{C},\partial)$ (resp., the complex $(\mathcal{C},\partial+\Phi)$) is an invariant of $L$ \cite[Thm. 1]{K} (resp., \cite[Thm. 4.2]{Lee}).

If we now remember the data of $\XX$, we obtain a third grading on the vector space underlying the Khovanov-Lee complex, as follows. Recall that there is a basis for $\mathcal{C}$ whose elements are in one-to-one correspondence with {\em oriented} Kauffman states (complete resolutions) of $\cP(L)$ (cf. \cite[Sec. 4.2]{JacoFest}). That is,  in the language of \cite{BN_Khov}, we identify a ``$v_+$" (resp., a ``$v_-$") marking on a component of a Kauffman state with a counter-clockwise (resp., clockwise) orientation on that component. We now obtain a third grading on the vector space underlying the complex by defining the ``$k$" grading of a basis element to be the algebraic intersection number of the corresponding oriented Kauffman state with any fixed oriented arc $\gamma$ from $\XX$ to $\OO$ that misses all crossings of $\cP(L)$. 

\begin{lemma} \label{lem:Leediffgr}The Lee differential decomposes into $4$ grading-homogeneous summands: \[\partial^{Lee} = (\partial_0 + \partial_-) + (\Phi_0 + \Phi_+)\] whose $(i,j,k)$ degrees are:
\begin{itemize}
	\item deg$(\partial_0) = (1,0,0)$
	\item deg$(\partial_-) = (1,0,-2)$
	\item deg$(\Phi_0) = (1,4,0)$
	\item deg$(\Phi_+) = (1,4,2)$
\end{itemize}
\end{lemma}

\begin{proof}
Lee proved that $\partial$ (resp., $\Phi$) are degree $(1,0)$ (resp., degree $(1,4)$) with respect to the $(i,j)$ gradings, so we need only verify that each of these endomorphisms splits into two pieces according to their $k$--gradings. The details of the splitting of $\partial$ is given in the proof of \cite[Lem. 1]{LRoberts} (our ``$j$" is Roberts' ``$q$" grading, and our ``$k$" is his ``$f$" grading).

To see that $\Phi$ also splits as claimed, recall (see \cite[Sec. 2]{LRoberts}) that component circles of a Kauffman state are either {\em trivial} (intersect the arc $\gamma$ from $\XX$ to $\OO$ in an even number of points) or {\em nontrivial} (intersect $\gamma$ in an odd number of points). Correspondingly, there are $3$ types of merge/split saddle cobordisms between pairs of components of a Kauffman state:
\begin{enumerate}
	\item {\em trivial} + {\em trivial} $\longleftrightarrow$ {\em trivial} (even + even = even)\\
	\item {\em trivial} + {\em nontrivial} $\longleftrightarrow$ {\em nontrivial} (even + odd = odd)\\
	\item {\em nontrivial} + {\em nontrivial} $\longleftrightarrow$ {\em trivial} (odd + odd = even)
\end{enumerate}

Recall from \cite[Sec. 4]{Lee} that for a merge cobordism, $\Phi$ is $0$ on all basis elements except the one for which both merging circles are labeled $v_-$ (Lee's {\bf 1} is our $v_+$ and {\bf x} is our $v_-$); this generator is sent to the one where the merged circle is marked with a $v_+$ (and the markings on all other circles are preserved). It follows that the $k$ degree of this map in cases (1), (2), and (3) above is $0$, $2$, and $2$ respectively. We leave to the reader the (similarly routine) check that the $k$ degrees of the split cobordism components of $\Phi$ in cases (1), (2), (3) are also $0$, $2$, and $2$.
\end{proof}

\begin{corollary}
The $j$ and $j-2k$ gradings on $\mathcal{C}$ endow $(\mathcal{C},\partial^{Lee})$ with the structure of a $(\Z \oplus \Z)$--filtered complex.
\end{corollary}

\begin{proof}
For each $(a,b) \in \Z \oplus \Z$, define \[\mathcal{F}_{a,b} :=  \mbox{Span}_{\F}\{{\bf x} \in \mathcal{C} \,\,|\,\, \mbox{gr}_{(j,j-2k)}({\bf x}) \geq (a,b)\}.\]

Lemma \ref{lem:Leediffgr} tells us that $\partial^{Lee}$ is non-decreasing with respect to the $j$ and $j-2k$ gradings, so $\mathcal{F}_{a,b}$ is a subcomplex for each $(a,b) \in (\Z \oplus \Z)$. Moreover, $(a',b') \geq (a,b) \in \Z \oplus \Z$ implies $\mathcal{F}_{a',b'} \subseteq \mathcal{F}_{a,b}$, as desired.
\end{proof}

\begin{definition} Let {\bf x} be a $(j,k)$--homogeneous basis element of $\mathcal{C}$, and let $t \in [0,2]$. Define \[j_t({\bf x}):= j({\bf x}) - t\cdot k({\bf x}).\]
\end{definition}

\begin{corollary}
For every $t \in [0,2]$, the $j_t$ grading endows $(\mathcal{C}, \partial^{Lee})$ with the structure of a (discrete, bounded) $\R$--filtered complex equipped with a finite filtered graded basis.
\end{corollary}

\begin{proof}
Lemma \ref{lem:Leediffgr} implies that $\partial^{Lee}$ is non-decreasing with respect to the $j_t$ grading for each $t \in [0,2]$. It follows that for each $a \in \R$, the subcomplexes \[\mathcal{F}_a := \mbox{Span}_{\F}\{{\bf x} \in \mathcal{C} \,\,|\,\, j_t({\bf x}) \geq a\}\] endow $\mathcal{C}$ with the structure of an $\R$--filtered complex. The finiteness of the distinguished filtered graded basis implies that the $\R$--filtration is discrete and bounded.
\end{proof}

\begin{remark} Note that the $j_t$--grading, for $t \in [0,2]$, does not exactly agree with the $\mbox{gr}_\theta$--grading coming from projecting to a line $\ell_\theta$ making an angle $\theta \in [0, \pi/2]$ with the horizontal axis, as described in Section \ref{sec:algprelim}. However, we have a bijective correspondence between values $t \in [0,2]$ and angles $\theta \in [0, \pi/2]$ given by the function $\theta(t) = tan^{-1}\left(\frac{t/2}{1- t/2}\right)$, and for a $\mbox{gr}_{(j_0, j_2)}$--homogeneous element ${\bf x} \in \mathcal{C}$, we have \[\mbox{gr}_{j_t}({\bf x}) = \sqrt{(1-t/2)^2 + (t/2)^2}\,\,\,\mbox{gr}_{\theta(t)}({\bf x}).\] Moreover, $\sqrt{(1-t/2)^2 + (t/2)^2} > 0$ for all $t \in [0,2]$, which tells us that for each $t \in [0,2]$ the $\R$--filtration induced by $\mbox{gr}_{j_t}$ is just a positive rescaling of the $\R$--filtration induced by $\mbox{gr}_{\theta(t)}$. So although the $\R$--filtrations are not precisely the same they are closely related.
\end{remark}

\begin{remark} \label{rmk:lattice} It will be convenient to plot the distinguished filtered graded basis elements of the annular Khovanov-Lee complex $(\mathcal{C}, \partial^{Lee})$ on the $\Z^2$ lattice in $\R^2$ with axes labeled $(j_0,j_2)$. Accordingly, we will often abuse notation and refer to the $j_t$-- or $k$--grading of a lattice point $(a,b) \in \Z^2$ when we mean the $j_t$-- or $k$--grading of a distinguished filtered graded basis element supported on $\{(a,b)\}$. In particular for $t \in [0,2]$, 
\[\mbox{gr}_{j_t}(a,b) = \left(1 - \frac{t}{2}\right)a + \left(\frac{t}{2}\right)b,\] and \[\mbox{gr}_k(a,b) = \frac{a-b}{2}.\]
\end{remark}

We are now ready to define the annular Rasmussen invariants of an oriented annular link. 

Recall that if $\cP(L) \subseteq S^2 - \OO \sim \R^2$ is a link diagram and $o$ is an orientation on $L$, then Lee \cite[Sec. 4]{Lee} (see also \cite{Rasmussen_Slice}) describes a canonical cycle $\mathfrak{s}_o \in \mathcal{C}(\cP(L))$ whose homology class $[\mathfrak{s}_o] \in H^{Lee}(L)$ is nonzero. Rasmussen used the $\Z$--filtration induced by the $j$ grading on $\mathcal{C}$ to define a knot invariant $s(K) \in 2\Z$ that is, essentially, the induced $j$ grading of this nonzero homology class:

\[s(K) := \mbox{gr}_j([\mathfrak{s}_o]) + 1 \in 2\Z\] Beliakova-Wehrli extended Rasmussen's definition to oriented links \cite{BW}:

\[s(L,o) := \mbox{gr}_j([\mathfrak{s}_o]) + 1 \in \Z.\]

\begin{remark} Beliakova-Wehrli's oriented link invariant is insensitive to orientation reversal. That is, if $\bar{o}$ is the orientation reverse of $o$, then $s(L,o) = s(L,\bar{o})$. Similarly, Rasmussen's knot invariant $s(K)$ does not depend on the orientation of $K$.
\end{remark} 

If $\cP(L) \subseteq S^2 - \{\OO, \XX\}$ is an {\em annular} link diagram and $o$ is an orientation on $L$, we have a discrete $\R$--filtration $\{\mathcal{F}^t(\mathcal{C}(\cP(L))_s\}_{s \in \R}$ associated to each $j_t$ grading, $t \in [0,2]$. Accordingly, we define the {\em annular Rasmussen invariants} as follows:

\begin{definition} \label{defn:annularRas} 
$d_t(L,o) := \mbox{gr}_{j_t}([\mathfrak{s}_o]) \in \R$
\end{definition}

There is also a natural involution $\Theta$ on $(\mathcal{C}, \partial^{Lee})$, previously described in a slightly different context in \cite[Prop.7.2, (3)]{QuiverAlgebras} and \cite[Lem. 2]{SchurWeyl}):

\begin{lemma} \label{lem:inv} Let $L \subset (A \times I) \subset S^3$ be an annular link, \[\cP(L) \subset (S^2 - \{\OO,\XX\}) \subset (S^2 - \{\OO\}) \sim \R^2\] a diagram for $L$, and \[\cP'(L) \subset (S^2 - \{\XX,\OO\}) \subset (S^2 - \{\XX\}) \sim \R^2\] the diagram obtained by exchanging the roles of $\OO$ and $\XX$. The corresponding map \[\Theta: (\mathcal{C}(\cP(L)), \partial^{Lee}) \rightarrow (\mathcal{C}(\cP'(L)), \partial^{Lee})\] is a chain isomorphism inducing an isomorphism \[H_*(\mathcal{C}(\cP(L), \partial^{Lee}) \cong H_*(\mathcal{C}(\cP'(L)), \partial^{Lee}).\]
\end{lemma}

\begin{proof}
Note that on generators of $\mathcal{C}$, the map $\Theta$ exchanges $v_{\pm}$ markings on {\em nontrivial} circles and preserves markings on {\em trivial} circles of each Kauffman state. See the proof of \cite[Lem. 2]{SchurWeyl}. The fact that $\Theta$ is a chain map on $(\mathcal{C}, \partial^{Lee})$ follows from \cite[Lem. 3]{SchurWeyl}. Since $\Theta^2 = \Id$, it is a chain isomorphism.
\end{proof}

\subsection{Annular braid closures, Plamenevskaya's invariant, and the $\sltwo$ action} \label{subsec:annbraid}
As previously mentioned, the annular Rasmussen invariants are particularly well-suited to studying annular braid closures equipped with their braid-like orientations. Explicitly, let $\sigma \in B_n$ be an $n$--strand braid and $\widehat{\sigma} \subset A \times I$ its annular closure. The braid-like orientation, $o_\uparrow$, of $\widehat{\sigma}$ is the one whose strands all wind positively around the braid axis. Its diagram winds counterclockwise about $\XX$ in $S^2 - \{\OO,\XX\}$.

When $L$ is an annular braid closure, the canonical Lee classes associated to the braid-like orientation $o_\uparrow$ and its reverse $o_\downarrow$ have nice descriptions in terms of Plamenevskaya's class \cite{Plam} and an $\sltwo$ action on the annular Khovanov-Lee complex \cite[Sec. 4]{SchurWeyl}. For the convenience of the reader, we briefly recall the relevant constructions here. In what follows, let $\mathcal{C}$ denote the $(i,j,k)$--graded vector space underlying the annular Khovanov-Lee complex associated to an oriented annular link. All relevant background and standard notation on the representation theory of the Lie algebra $\sltwo$ can be found in \cite[Sec. 2]{SchurWeyl}.

As in \cite[Sec. 4]{SchurWeyl}, let 
\begin{itemize}
\item $V := \mbox{Span}_\C\{v_+,v_-\}$ denote the defining representation of $\sltwo$, with gr$_{(j,k)}(v_{\pm}) = (\pm 1, \pm 1)$,
\item $V^* := \mbox{Span}_\C\{\overline{v}_+, \overline{v}_-\}$ its dual (where $\overline{v}_{\pm} := v_{\mp}^*)$, with gr$_{(j,k)}(\overline{v}_{\pm}) = (\pm 1, \pm 1)$, and
\item $W := \mbox{Span}_\C\{w_+,w_-\}$ be the trivial two-dimensional representation, with gr$_{(j,k)}(w_{\pm}) = (\pm 1, 0)$
\end{itemize}

Now let $K \subset S^2 - \{\OO,\XX\}$ be a Kauffman state (complete resolution) in the cube of resolutions of a diagram of $L \subset (A \times I)$, and suppose $K$ has $\ell_n$ nontrivial circles and $\ell_t$ trivial circles. Choose any ordering  $C_1, \ldots, C_{\ell_n}, C_{\ell_n + 1}, \ldots, C_{\ell_n + \ell_t}$ of the circles so all of the nontrivial circles are listed first. For $i \in \{1, \ldots, \ell_n\}$ let $X(C_i) \in \{0, \ldots \ell_n - 1\}$ denote the number of nontrivial circles of $K$ lying in the same component of $S^2 - C_i$ as $\XX$ and define \[\epsilon(C_i) := (-1)^{X(C_i)}.\] Then we assign to the Kauffman state $K$ the $\sltwo$ representation: \[\left(\bigotimes_{\epsilon(C_i) = 1} V\right) \otimes \left(\bigotimes_{\epsilon(C_i) = -1} V^*\right) \otimes \left(\bigotimes_{s=1}^{\ell_t} W\right),\] with $(i,j)$--grading shifts as described in \cite{K}, \cite{BN_Khov}. In \cite{SchurWeyl} it is shown that the $\sltwo$ action on $\mathcal{C}$ commutes with a particular summand of the Lee differential and hence can be used to endow the annular Khovanov homology of an annular link the structure of an $\sltwo$ representation.

Now suppose $L$ is the (oriented) annular braid closure of a braid $\sigma \in \Braid_n$. Then it is implicit in \cite{Plam} (see also \cite{GW_ColJones}, \cite[Prop. 2.5]{GN}) that the graded vector space underlying the Khovanov-Lee complex has a {\em unique} $\sltwo$ irrep of highest weight $n$ and, hence (since the ``$k$" grading on $\mathcal{C}$ coincides with the $\sltwo$ weight space grading) there is a unique generator of $\mathcal{C}$ whose ``$k$" grading is $-n$. This is the lowest-weight vector in the unique $(n+1)$--dimensional irreducible subrepresentation of $\mathcal{C}$. It is the distinguished basis element corresponding to marking each of the circles in the all-braidlike resolution of $\widehat{\sigma}$ (i.e., the oriented resolution for the braid-like orientation) with a $v_-$. This is the class Plamenevskaya denotes by $\tilde{\psi}(\widehat{\sigma})$ in \cite{Plam}. We will denote it by ${\bf v}_-$.

\begin{remark} Plamenevskaya shows that ${\bf v}_-$ is a cycle in the {\em Khovanov} complex $(\mathcal{C}, \partial)$, and (up to multiplication by $\pm 1$) its associated homology class $[{\bf v}_-] \in \Kh(\widehat{\sigma}) = H_*(\mathcal{C},\partial)$ is an invariant of the transverse isotopy class of $\widehat{\sigma}$. Note that ${\bf v}_-$ is {\em not} a cycle in the {\em Lee} complex $(\mathcal{C}, \partial + \Phi)$. 
\end{remark}

The canonical Lee classes associated to the two braid-like orientations have nice descriptions in terms of the Plamenevskaya class and the $\sltwo$--module structure. In the following, recall that $e^{(k)} := \frac{e^k}{k!}$ is the so-called {\em $k$--th divided power} of $e \in \sltwo$:

\begin{proposition} Let $\sigma \in B_n$, and let $\widehat{\sigma}$ be its an annular braid closure. If $n$ is even,
\begin{eqnarray*}
	\mathfrak{s}_{o_\uparrow}(\widehat{\sigma}) &=& \sum_{k=0}^n e^{(k)} {\bf v}_-\\
	\mathfrak{s}_{o_\downarrow}(\widehat{\sigma}) &=& \sum_{k=0}^n (-1)^{k} e^{(k)} {\bf v}_-
\end{eqnarray*}
If $n$ is odd, the identifications of $\mathfrak{s}_{o_\uparrow}$ and $\mathfrak{s}_{o_\downarrow}$ with the summations on the right are reversed.
\end{proposition}

\begin{proof} By definition (cf. \cite[Sec. 4]{Lee} and \cite[Sec. 2.4]{Rasmussen_Slice}), we see that $\mathfrak{s}_{o_\uparrow}(\widehat{\sigma})$ is the cycle in the Lee complex where the outermost circle in the all-braidlike resolution has been marked with a {\bf b}, and the remaining circles are marked alternatingly with {\bf a}'s and {\bf b}'s from outermost to innermost. We now compare this description with the definition of the $\sltwo$ action on $\mathcal{C}$ as a tensor product representation of copies of the defining representation $V$ and its dual $V^*$, where circles are marked alternatingly with $V$ and $V^*$ from innermost to outermost. Indeed, recalling that 
\begin{itemize}
	\item ${\bf a} = v_- + v_+$ and ${\bf b} = v_- - v_+$ and
	\item in the defining representation $V$, $e(v_-) = v_+$, while in its dual $V^*$, $e(\overline{v}_-) = -\overline{v}_+$
\end{itemize} we see that if $n$ is even, we have \[s_{o_\uparrow}(\widehat{\sigma}) = (v_- + ev_-) \otimes (\overline{v}_- + e\overline{v}_-) \otimes \ldots \otimes (\overline{v}_- + e\overline{v}_-)\] while if $n$ is odd, we have \[s_{o_\uparrow}(\widehat{\sigma}) = (v_- - ev_-) \otimes (\overline{v}_- - e\overline{v}_-) \otimes \ldots \otimes (v_- - ev_-),\] where in the above, the tensor product factors from left to right correspond to markings of circles from innermost to outermost.

We now need a small bit of notation. Suppose $S$ is a $k$--element subset of $\{1, \ldots, n\}$, and $V_1 \otimes \ldots \otimes V_n$ is an $n$--factor tensor product representation of $\sltwo$. Then we will denote by $E_S$ the map that sends a decomposable vector $v_1 \otimes \ldots \otimes v_n$ to the decomposable vector $w_1 \otimes \ldots \otimes w_n$, where \[w_i = \left\{\begin{array}{cl} ev_i & \mbox{if $i \in S$}\\v_i & \mbox{if $i \not\in S$.}\end{array}\right.\]

When $n$ is even, we then see that \[\mathfrak{s}_{o_\uparrow}(\widehat{\sigma}) = \sum_{k=0}^n\,\,\sum_{\substack{S \subseteq \{1, \ldots, n\},\\ |S| = k}} E_S({\bf v}_-),\] and when $n$ is odd, we have \[\mathfrak{s}_{o_\uparrow}(\widehat{\sigma}) = \sum_{k=0}^n (-1)^k \sum_{\substack{S \subseteq \{1, \ldots, n\},\\ |S| = k}} E_S({\bf v}_-).\]

But it follows from the definition of the tensor product representation and the fact that $e^2(v_-) = 0$ (resp., $e^2(\overline{v}_-) = 0$) in $V$ (resp.,  $V^*$) that \[e^k({\bf v}_-) = k! \sum_{\substack{S \subseteq \{1, \ldots, n\},\\ |S| = k}} E_S({\bf v}_-),\] which tells us that when $n$ is even (resp., odd), we have $\mathfrak{s}_{o_\uparrow} = \sum_{k=0}^n e^{(k)}({\bf v}_-)$ (resp., $\mathfrak{s}_{o_\uparrow} = \sum_{k=0}^n (-1)^k e^{(k)}({\bf v}_-)$.)

Since $\mathfrak{s}_{o_\downarrow}$ is obtained from $\mathfrak{s}_{o_\uparrow}$ by replacing all ${\bf a}$ markings with ${\bf b}$ and vice versa, the desired conclusion follows.
\end{proof}

\section{Proof of Main Theorem} \label{sec:mainthm}
We are now ready to state and prove a detailed version of our main theorem.

\begin{theorem} \label{thm:main} Let $L \subset (A \times I)$ be an annular link with wrapping number $\omega$, let $o$ be an orientation on $L$, and let $t \in [0,2]$.
\begin{enumerate}
	\item $d_t(L,o)$ is an oriented annular link invariant.
	\item $d_{1-t}(L,o) = d_{1+t}(L,o)$ for all $t \in [0,1]$.
	\item $d_0(L,o) = d_2(L,o) = s(L,o) - 1.$
	\item Viewed as a function $[0,2] \rightarrow \mathbb{R}$, $d_t(L,o)$ is piecewise linear. Moreover, letting 
		\[m_{t}(L,o) := \mbox{lim}_{\epsilon \rightarrow 0^+} \frac{d_{t+\epsilon}(L,o) - d_{t}(L,o)}{\epsilon}\] denote the (right-limit) slope at $t$, we have $m_t(L,o) \in \{-\omega, -\omega+2, \ldots, \omega-2, \omega\}$ for all $t \in [0,2).$
	\item Suppose $(L,o)$ and $(L',o')$ are non-empty oriented annular links, and $F$ is an oriented cobordism from $(L,o)$ to $(L',o')$ for which each component of $F$ has a boundary component in $L$. Then if $F$ has $a_0$ even-index annular critical points, $a_1$ odd-index annular critical points, and $b_0$ even-index non-annular critical points, then \[d_t(L,o) - d_t(L',o') \leq (a_1-a_0) - b_0(1-t).\]
	\item Suppose $(L,o)$, $(L',o')$, and F are as in (5) above, and suppose that in addition each component of $F$ has a boundary component in $L'$. Then \[|d_t(L,o) - d_t(L',o')| \leq (a_1-a_0) - b_0(1-t).\]
\end{enumerate}
\end{theorem}

The following proposition will be crucial to the proof of our main theorem. See the beginning of Section \ref{sec:topprelim} for a discussion of annular and non-annular elementary string interactions and \cite{Lee}, \cite{Rasmussen_Slice} for a description of the chain maps on the Lee complex associated to each of these elementary string interactions.

\begin{proposition} \label{prop:cobdegree} For $t \in [0,2]$, the filtered $j_t$ degree of the Lee chain map associated to 
\begin{enumerate}
	\item an annular elementary saddle cobordism (odd index critical point) is $-1$,
	\item an annular birth/death (even index critical point) is $1$,
	\item a non-annular birth/death (even index critical point) is $1-t$, and
	\item an annular Reidemeister move (product cobordism) is $0$.
\end{enumerate}
\end{proposition}

\begin{proof} [Proof of Proposition \ref{prop:cobdegree}]
\begin{enumerate}
	\item Recall (\cite{Lee}, \cite[Sec. 2.4]{Rasmussen_Slice}) that the Lee chain map associated to a saddle cobordism agrees with the split/merge map involved in the definition of the Lee differential. (The only difference is that the $+1$ $j$--grading shift for each successive (co)homological grading in the Lee complex ensures that the Lee differential has filtered degree $0$, not $-1$.) We can now appeal to Lemma \ref{lem:Leediffgr}: depending on the configuration of the splitting/merging circles, the $j_t$ degree of the associated map is $-1$ (resp., $-1 + 2t, 3,$ or $3-2t$) for configurations contributing to $\partial_0$ (resp., to $\partial_-, \Phi_0,$ or $\Phi_+$). Since each of these degrees is bounded below by $-1$ on the interval $t \in [0,2]$, we conclude that a saddle cobordism map has filtered degree $-1$.
	\item The Lee chain map associated to a cup is the map $\iota$ from \cite{Lee} ($\iota'$ from \cite{Rasmussen_Slice}), and the map associated to a cap is the map $\epsilon$ from \cite{Lee} ($\epsilon'$ from \cite{Rasmussen_Slice}). If the cup/cap is {\em annular}, the $(j,k)$ degree of the map is $(1,0)$, hence the $j_t$ degree is $1$, as desired.
	\item If the cup/cap is {\em non-annular}, the $(j,k)$ degree of the map is $(1,1)$, hence the $j_t$ degree is $1-t$, as desired.
	\item The Lee chain maps associated to Reidemeister moves--each of which induces an isomorphism on Lee homology--are given in \cite{Lee} and \cite[Sec. 6]{Rasmussen_Slice}. If the Reidemeister move is {\em annular}, then none of the local diagrams contain either basepoint $\XX$ or $\OO$. Noting that each of these maps is therefore a linear combination of compositions of the annular cobordism maps described in parts (1) and (2) above (cf. \cite[Proof of Thm. 1]{BN_Tangles}), it is then straightforward to check in each case that the filtered $j_t$ degree of each of the annular Reidemeister maps is $0$. As an example, consider the map $\rho_1'$ described in \cite[Sec. 6]{Rasmussen_Slice}. If the RI move is annular, the circle on the RHS of \cite[Fig. 6]{Rasmussen_Slice} is trivial, and hence the lowest degree term with respect to the $j_t$ grading is always $0$. Note that there are higher order terms of degree $4-2t$ and $2t$ when the arc in the local diagram on the LHS belongs to a nontrivial circle, and there is a higher order term of degree $4$ when the arc belongs to a trivial circle, but in both cases the lowest degree term is $0$.\end{enumerate}
\end{proof}

\begin{proof} [Proof of Theorem \ref{thm:main}]
\begin{enumerate}
	\item Two diagrams representing isotopic (oriented) annular links will be related by a finite sequence of annular isotopies and annular Reidemeister moves: that is, isotopies and Reidemeister moves of the diagram that never cross either of the marked points $\XX$, $\OO$. By  Lemma \ref{lem:ZplusZtoR}, it will therefore suffice to show that the chain maps on the Lee complex associated to each annular Reidemeister move induces a $\Z \oplus \Z$--filtered chain homotopy equivalence. 
	
In \cite[Proof of Thm. 1]{BN_Tangles}, Bar-Natan explicitly defines chain maps and chain homotopies yielding homotopy equivalences associated to each Reidemeister move of the formal complexes associated to links over the category $Cob^3_{/l}(\emptyset)$. These yield chain homotopy equivalences of the corresponding Lee complexes, by viewing Bar-Natan's maps as morphisms in the category $Cob_1(\emptyset)$ described in \cite[Sec. 2]{BN_Morr} and then applying the functor $\mbox{Hom}(\emptyset, -)$. Moreover, after applying the functor, the chain maps agree with those defined by Lee \cite{Lee} (see also \cite[Sec. 6]{Rasmussen_Slice}). We already confirmed in Proposition \ref{prop:cobdegree}, part (4), that these chain maps are, indeed, degree $0$ with respect to the $j_t$ grading, for $t = 0,2$. 

It therefore remains to verify that each of the chain {\em homotopies} Bar-Natan defines in this way has filtered degree $0$ with respect to the $j_t$ grading, for $t=0,2$. Just as in the proof of Proposition \ref{prop:cobdegree}, the relevant observation here is that if a Reidemeister move is {\em annular}, then none of the local diagrams in \cite[Figs. 5,6,7]{BN_Tangles} contain either basepoint $\XX$ or $\OO$. This will insure that the lowest degree term of each chain homotopy with respect to the $j_t$ grading is $0$.

For example the chain homotopies associated to the annular Reidemeister II move, pictured in \cite[Fig. 7]{BN_Tangles}, are scalar multiples of the maps $\iota'$ and $\epsilon'$, respectively. They therefore have $j$ degree $0$ (once we account for the $+1$ $j$--grading shifts in the successive homological gradings of the complex) and $k$ degree $0$ (because the $\XX$ and $\OO$ basepoints are not present in the local diagram).

Note that Bar-Natan's proof of the homotopy equivalence of complexes related by a Reidemeister III move is slightly indirect, but again each of the maps involved in \cite[Lem. 4.4, 4.5]{BN_Tangles} have filtered $j_t$ degree $0$ for $t = 0,2$ when the Reidemeister III move is annular, since the basepoints $\OO$ and $\XX$ are absent from the local diagrams. 
	\item The involution $\Theta$ described in Lemma \ref{lem:inv} induces a $\Z/2\Z$ symmetry on $\mathcal{C}^{Lee}$ that exchanges the roles of the $j_{1-t}$ and $j_{1+t}$ gradings for all $t \in [0,1]$. That is, for each distinguished filtered graded basis element ${\bf x} \in \mathcal{C}^{Lee}$, $\Theta({\bf x})$ is also a distinguished filtered graded basis element, and $\mbox{gr}_{j_{1-t}}(\Theta({\bf x})) = \mbox{gr}_{j_{1+t}}({\bf x})$ for all $t \in [0,1]$. This fact follows from (the proof of) \cite[Lem. 2]{SchurWeyl}, noting (cf. \cite[Rmk. 2]{SchurWeyl}) that the $j'$ grading there matches our $j_1$ grading.\footnote{Indeed, recalling that $\mathcal{C}^{Lee}$ is an $\sltwo$ representation, the $\Z/2\Z$ symmetry on $\mathcal{C}^{Lee}$ is a manifestation of the symmetry in $\sltwo$ coming from exchanging the roles of $e$ and $f$. See \cite[Lem. 5]{SchurWeyl}.} Moreover, from the definitions of $\Theta$ and the canonical Lee generator $\mathfrak{s}_o$ associated to an orientation $o$ one readily verifies that $\Theta(\mathfrak{s}_o) = \pm \mathfrak{s}_o$. 

It follows that if $\eta \in \mathcal{C}^{Lee}(L,o)$ (resp., $\Theta(\eta)$) is a cycle in $\mathcal{C}^{Lee}(L,o)$ representing $[\mathfrak{s}_o]$ and realizing its $j_{1-t}$ grading (that is, $j_{1-t}(\eta) = j_{1-t}[\mathfrak{s}_o]$), then $\Theta(\eta) \in \mathcal{C}^{Lee}(L,o)$ (resp., $\eta$) also represents $\pm [\mathfrak{s}_o]$ and realizes its $j_{1+t}$ grading. So $d_{1-t}(L,o) = d_{1+t}(L,o)$ for all $t \in [0,1]$, as desired.
	\item Let $(L,o) \subset A \times I \subset S^3$ be an oriented link. By definition (cf. \cite[Defn. 3.4]{Rasmussen_Slice}, \cite[Sec. 6]{BW}), \[s(L,o) = \frac{1}{2}\left(\mbox{gr}_j[\mathfrak{s}(L,o)-\mathfrak{s}(L,\overline{o})] + \mbox{gr}_j[\mathfrak{s}(L,o)+\mathfrak{s}(L,\overline{o}]\right).\] Moreover, we know:
	\begin{itemize}
		\item $\left|\mbox{gr}_j[\mathfrak{s}(L,o)-\mathfrak{s}(L,\overline{o})] - \mbox{gr}_j[\mathfrak{s}(L,o)+\mathfrak{s}(L,\overline{o}]\right| = 2$ and
		\item $\mbox{gr}_j[\mathfrak{s}(L,o)] = \mbox{gr}_j[\mathfrak{s}(L,\overline{o})] =\min\left\{\mbox{gr}_j[\mathfrak{s}(L,o)-\mathfrak{s}(L,\overline{o})],  \mbox{gr}_j[\mathfrak{s}(L,o)+\mathfrak{s}(L,\overline{o}]\right\}.$
	\end{itemize}
	So $d_0(L,o) = s(L,o) - 1$, as desired. It follows from part (2) above that $d_2(L,o) = s(L,o) - 1$ as well.
	\item Note that $\mathcal{C}^{Lee}$ has a distinguished finite basis $\mathcal{B}$ that is homogeneous with respect to the $(j,k)$ gradings, hence with respect to the $(j_0, j_2)$ gradings. Moreover, since the Lee differential is monotonic (non-decreasing) with respect to the $(j_0,j_2)$--grading, $\mathcal{B}$ is a filtered graded basis for an induced $\Z \oplus \Z$--filtration on $\mathcal{C}^{Lee}$. As in Section \ref{sec:algprelim} we can therefore plot the generators of $\mathcal{C}^{Lee}$ on $\Z \oplus \Z \subset \R^2$, with axes labeled $(j_0, j_2)$ and consider, for the canonical Lee class $[\mathfrak{s}_o] \neq 0 \in H_*(\mathcal{C}^{Lee})$, the set $\mathbb{S}([\mathfrak{s}_o])$ of subsets of $S \subset \Z \oplus \Z$ for which there exists a representative of $[\mathfrak{s}_o]$ supported on $S$ (see Defns. \ref{defn:grsupport}, \ref{defn:subseteta}). Since $\mathcal{B}$ is finite, $\mathbb{S}(\eta)$ is necessarily finite for each $\eta \neq 0 \in H_*(\mathcal{C}^{Lee})$.

Exactly as in Lemma \ref{lem:latticegraddef},\footnote{Note that we have plotted ${\bf x} \in \mathcal{B}$ on the $(j_0,j_2)$ lattice, so if ${\bf x} \in \mathcal{B}$ is associated to lattice point $(a,b) \in \Z \oplus \Z$, then $\mbox{gr}_j({\bf x}) = a$ and $\mbox{gr}_k({\bf x}) = \frac{a-b}{2}.$} we see that for each $\eta \neq 0 \in H_*(\mathcal{C}^{Lee})$ and each $t \in [0,2]$, 
\[\mbox{gr}_{j_t}(\eta) = \max_{S \in \mathbb{S}(\eta)}\left\{\min_{(a,b) \in S}\left\{a - t\left(\frac{a-b}{2}\right)\right\}\right\}.\]

Now fix $\eta \neq 0 \in H_*(\mathcal{C}^{Lee})$ and consider the behavior of $\mbox{gr}_{j_t}(\eta)$ as $t \in [0,2]$ varies. For each $(a,b) \in \Z \oplus \Z$, the function $a - t\left(\frac{a-b}{2}\right)$ is linear in $t$. Moreover, each $S \in \mathbb{S}(\eta)$ contains finitely many $(a,b) \in \Z \oplus \Z$, so the function $\min_{(a,b) \in S}\{a - t\left(\frac{a-b}{2}\right)\}$ is piecewise linear in $t \in [0,2]$ with finitely many discontinuities. Finally, there are finitely many $S \in \mathbb{S}(\eta)$, so the function $\max_{S \in \mathbb{S}(\eta)}\left\{\min_{(a,b) \in S}\left\{a - t\left(\frac{a-b}{2}\right)\right\}\right\}$ is also piecewise linear with finitely many discontinuities. This proves the first part of the statement.

To prove the second part of the statement, we note that by the discussion above, $m_t(L,o)$ is always equal to $-\frac{a-b}{2}$ for some $(a,b) \in S$, $S \in \mathbb{S}([\mathfrak{s}_o])$. But any chain ${\bf x} \in \mathcal{C}^{Lee}$ supported in lattice point $(a,b) \in \Z^2$ has $\mbox{gr}_k({\bf x}) = \frac{a-b}{2}.$ Moreover, for a link $L$ with wrapping number $\omega$, the $k$--gradings of distinguished basis elements of $\mathcal{C}^{Lee}(L,o)$ lie in the set $\{-\omega, -\omega+2, \ldots, \omega-2,\omega\}$ (cf. \cite[Sec. 4.2]{JacoFest}). This proves the second statement.

	\item Suppose $F$ is an orientable cobordism from $(L,o)$ to $(L',o')$ for which each component of $F$ has a boundary component in $(L,o)$.  Then $F$ has no closed components, and the orientation on $F$ is uniquely determined by the orientation $o$ on $L$. Hence, by \cite[Prop. 4.1]{Rasmussen_Slice}, the map $\phi_F$ on Lee homology induced by $F$ sends $[\mathfrak{s}_o]$ to a nonzero multiple of $[\mathfrak{s}_{o'}]$. Now fix $t \in [0,2]$ and let ${\bf x} \in \mathcal{C}^{Lee}(L)$ be a cycle representing $[\mathfrak{s}_o]$ for which $\mbox{gr}_{j_t}({\bf x}) =\mbox{gr}_{j_t}[\mathfrak{s}_o]$. Perturb $F$ slightly so it is Morse, with $a_0$ even-index annular critical points, $a_1$ odd-index annular critical points, and $b_0$ even-index non-annular critical points.\footnote{Note that by transversality all odd-index critical points are necessarily annular.}  By Proposition \ref{prop:cobdegree}, we then have \[\mbox{gr}_{j_t}(\phi_F({\bf x})) = d_t(L,o) + a_0 - a_1 + b_0(1-t),\] so since $[\phi_F({\bf x})] = c[\mathfrak{s}_{o'}]$ for $c \neq 0$, we have $\mbox{gr}_{j_t}(\phi_F({\bf x})) \leq \mbox{gr}_{j_t}[\mathfrak{s}_{o'}] = d_t(L',o')$, so \[d_t(L,o) - d_t(L',o') \leq (a_1-a_0) - b_0(1-t).\]
	\item Apply part (5) to the cobordism from $(L',o')$ to $(L,o)$ obtained by running $F$ in reverse. Since the parity of the index of the critical points doesn't change, we get the other inequality: \[d_t(L',o') - d_t(L,o) \leq (a_1-a_0) - b_0(1-t).\]

\end{enumerate}

\end{proof}

\begin{remark} \label{rmk:kgrading} Implicit in the proof of part (4) of Theorem \ref{thm:main} is the observation that $m_t(L,o)$ is $-1$ times the $k$--grading of a lattice point $(a,b)$ that determines the $j_t$ grading of $[\mathfrak{s}_o]$. That is, \[m_t(L,o) = -\frac{a-b}{2}\] for a lattice point $(a,b)$ satisfying  $\mbox{gr}_{j_t}(a,b) = \mbox{gr}_{j_t}[\mathfrak{s}_o]$. Moreover, $(a,b) \in S$ for some $S \in \mathbb{S}([\mathfrak{s}_o])$, and $\mbox{gr}_{j_t}(a,b)$ is minimum among all lattice points in $S$.
\end{remark}	

As we will have particular interest in annular and non-annular cobordisms between annular braid closures, we recall the following definition, made by Hughes, generalizing a notion due to Kamada \cite{Kamada} and Viro \cite{Viro} (see also Rudolph \cite{Rudolph}).

\begin{definition} \cite{Hughes} \label{defn:braidedcob} A {\em braided cobordism} is a smoothly, properly embedded surface \[F \subset S^3 \times [0,1]\] on which the projection $pr_2: S^3 \times [0,1] \rightarrow [0,1]$ restricts as a Morse function with each regular level set $F \cap (S^3 \times \{t\})$ a closed braid in $S^3 \times \{t\}$.
\end{definition}

\begin{remark} As noted in \cite[Sec. 2.2]{Hughes}, a {\em singular still} of a {\em movie presentation} of a braided cobordism $F$ will change the diagram by one of:
\begin{enumerate}
	\item Addition or deletion of a single loop around $\XX$ disjoint from the rest of the diagram,
	\item Addition or deletion of a single crossing between adjacent strands in the braid diagram,
	\item A single braid-like Reidemeister move of type II or III.
\end{enumerate}
\end{remark}

\begin{definition} We will say that a braided cobordism $F$ from $\widehat{\sigma}_0 = F \cap (S^3 \times \{0\})$ to $\widehat{\sigma}_1 = F \cap (S^3 \times \{1\})$ is {\em braid-orientable} if it admits an orientation compatible with the braid-like orientations of $\widehat{\sigma}_0$ and $\widehat{\sigma}_1$.
\end{definition}

\begin{corollary} \label{cor:braidedcob}
\begin{enumerate}
	\item If $\sigma_0 \in B_{n_0}$ and $\sigma_1 \in B_{n_1}$ are braids ($n_0 \in \Z^+, n_1 \in \Z^{\geq 0}$), and $F$ is a braid-orientable braided cobordism from $\widehat{\sigma}_0$ to $\widehat{\sigma}_1$ with $a_1$ odd-index (annular) critical points and $b_0$ even-index (non-annular) critical points, and  each component of $F$ has a component in $\sigma_0$, then \[d_t(\widehat{\sigma}_0) - d_t(\widehat{\sigma}_1) \leq a_1 - b_0(1-t).\]
	\item If $\sigma_0, \sigma_1$ are as above, and in addition each component of $F$ has a component in $\sigma_1$, then \[|d_t(\widehat{\sigma}_0) - d_t(\widehat{\sigma}_1)| \leq a_1 - b_0(1-t).\]
\end{enumerate}
\end{corollary}

Note that Hughes has proved in \cite{Hughes} that every link cobordism between braid closures is isotopic rel boundary to a braided cobordism.

In \cite{Rudolph}, Rudolph studies oriented ribbon-immersed surfaces in $S^3$ bounded by braids realizable as products of {\em bands}. Recall that algebraically a {\em band} is a {\em conjugate} of an elementary Artin generator, and a band presentation of a braid $\sigma$ is an explicit decomposition: \[\sigma = \prod_{j=1}^c \omega_j \sigma_{i_j}^{\pm} (\omega_j)^{-1}.\] The nomenclature is justified by Rudolph's observation that if $\sigma \in \Braid_n$ can be written as a product of $c$ bands, then the closure of $\sigma$ bounds an obvious ribbon-immersed surface in $S^3$  with one disk ($0$--handle) for each strand of the braid and one band ($1$--handle) for each term in the product. There is then an obvious push-in to $B^4$ of this surface that is {\em ribbon}--i.e., Morse with respect to the radial function with no critical points of index $2$.

Recall (see \cite{Rudolph}) that the {\em band rank} of a braid (conjugacy class) $\sigma \in \Braid_n$, denoted $\mbox{rk}_n(\sigma)$ is the smallest $c \in \Z^{\geq 0}$ for which $\sigma$ can be expressed as a product of $c$ conjugates of elementary Artin generators (either positive or negative). That is, letting $\sigma_k$ denote the $k$th elementary Artin generator, \[\mbox{rk}_n(\sigma) := \mbox{min}\left\{c \in \Z^{\geq 0} \,\,\left|\,\, \sigma = \prod_{j=1}^c \omega_j \sigma_{i_j}^\pm (\omega_j)^{-1} \mbox{ for some $\omega_j, \sigma_{i_j} \in \Braid_n$.}\right.\right\}\] 

We obtain:

\begin{corollary} \label{cor:bandrank} Let $\sigma \in \Braid_n$. Then \[\mbox{max}_{t \in [0,2]} \left|d_t(\widehat{\sigma}) - d_t(\widehat{\Id}_n)\right| \leq rk_n(\sigma).\] Noting that $d_t(\widehat{\Id}_n) = -|n(1-t)|$, this bound can be rewritten: \[\mbox{max}_{t \in [0,2]} \left|d_t(\widehat{\sigma}) + |n(1-t)|\right| \leq rk_n(\sigma).\]
\end{corollary}

\begin{proof} If $\mbox{rk}_n(\sigma) = c$, there is an obvious braided cobordism from $\widehat{\sigma}$ to $\widehat{\Id}_n$ with $c$ odd-index critical points. This is the cobordism whose movie performs the oriented resolution at each of the elementary Artin generators anchoring each of the $c$ bands of $\sigma$. It is oriented compatibly with the braid-like orientations of $\widehat{\sigma}$ and $\widehat{\Id}_n$, and every component of the cobordism has boundary in both $\widehat{\sigma}$ and $\widehat{\Id}_n$. Appealing to Corollary \ref{cor:braidedcob} we then obtain \[\left|d_t(\widehat{\sigma}) - d_t(\widehat{\Id}_n) \right| \leq rk_n(\sigma)\] for each $t \in [0,2]$, hence \[\mbox{max}_{t \in [0,2]}\left|d_t(\widehat{\sigma}) + |n(1-t)|\right| \leq rk_n(\sigma),\] as desired. 
\end{proof}

\begin{remark} Note that when the maximum value of the distance of $d_t(\widehat{\sigma})$ occurs at $t = 0,1, 2$, then the above bound on band rank mostly follows from previously known results. In particular, we know that the absolute value of the {\em writhe} of $\sigma$ is a lower bound for $\mbox{rk}_n(\sigma)$. The fact that $d_1(\widehat{\sigma})$ agrees with the writhe of $\widehat{\sigma}$ then gives the bound at $t=1$. By work of Rasmussen \cite{Rasmussen_Slice} and Beliakova-Wehrli \cite{BW} we know that if $o$ is an orientation on a link $L$ and $\chi(L)$ is the maximal Euler characteristic among all smoothly imbedded orientable surfaces $S \subset B^4$
\begin{itemize}
	\item bounded by $L$, 
	\item oriented compatibly with the orientation $o$, and 
	\item for which every connected component of $S$ has a boundary in $L$, 
\end{itemize}
then $s(L,o) -1 \leq -\chi(L)$. Then by Rudolph's construction of a ribbon-imbedded surface of Euler characteristic $n-c$ when $\sigma \in \Braid_n$ is written as the product of $c$ bands we obtain the bound $d_0(\widehat{\sigma}) = s(\widehat{\sigma}) - 1 \leq -n+c$.
\end{remark}

We close this section by noting that in \cite[Sec. 3]{Rudolph} Rudolph proves the amazing fact that if $S$ is a ribbon-immersed orientable surface in $S^3$, then it is isotopic to a {\em banded, braided surface} of Euler characteristic $n-c$ constructed as above (from the closure of an $n$--strand braid realized as the product of $c$ bands). Rudolph's result has the following interesting corollary:

\begin{corollary} \cite{Rudolph} If $K \subset S^3$ is {\em ribbon} (i.e., it bounds a smoothly imbedded disk in $B^4$, Morse, with no maxima) then for some $n \in \N$ there exists a braid $\sigma \in \Braid_n$ with $K = \widehat{\sigma}$ and $\mbox{rk}_n(\sigma) = n-1$.  
\end{corollary}

Rudolph's result suggests that if one has a braid conjugacy class invariant (like ours) yielding a lower bound on band rank, along with a concrete understanding of how the invariant changes under Markov stabilization and destabilization, then one may be able to extract from it an effective ribbon obstruction.

\section{Properties of annular Rasmussen invariant and applications} \label{sec:props}

In this section, we will focus our attention on annular braid closures. In particular, we use properties of the annular Rasmussen invariants $d_t(\widehat{\sigma}, o_\uparrow)$, which we shall abbreviate to $d_t(\widehat{\sigma})$, to study braids as mapping classes. We begin with a few definitions.

Let $D_n$ denote the standard closed unit disk in $\C$ equipped with a set, $\Delta = \{p_1, \ldots, p_n\}$, of $n$ distinct marked points in $D_n \setminus \partial D_n$ which for convenience we assume to be arranged in increasing order of index from left to right on the real axis. We will denote by $\sigma_i$ the standard positive elementary Artin generator of $\Braid_n$. Subject to the identification of $\Braid_n$ with $\pi_0(\mbox{Diff}^+(D_n))$, the mapping class group of $D_n$, $\sigma_i$ is the positive (counterclockwise) half-twist about a regular neighborhood of the subarc of the real axis joining $p_i$ to $p_{i+1}$.

\begin{definition} \label{defn:admissible} We will say that $\gamma: [0,1] \rightarrow D_n$ is an {\em admissible} arc from $\gamma(0)$ to $\gamma(1)$ if it satisfies
\begin{enumerate}
	\item $\gamma$ is a smooth imbedding transverse to $\partial D_n$,
	\item $\gamma(0) \in \{-1, p_1, \ldots, p_n\}$ and $\gamma(1) \in \{-1, p_1, \ldots, p_n\} \setminus \{\gamma(0)\}$,
	\item $\gamma(t) \in D_n \setminus (\partial D_n \cup \{p_1, \ldots, p_n\})$ for all $t \in (0,1)$, and
	\item $\frac{d\gamma}{dt} \neq 0$ for all $t \in [0,1]$.
\end{enumerate}
\end{definition}

We will often abuse notation and use $\gamma$ to refer to the image of $\gamma$ in $D_n$. We also note that via its identification with the mapping class group of $D_n$, $\Braid_n$ acts on the set of isotopy classes of admissible arcs. If $\gamma$ represents an isotopy class of admissible arcs, we will denote the image of $\gamma$ under $\sigma$ by $(\gamma)\sigma$.

\begin{definition} Two admissible arcs $\gamma, \gamma'$ are said to be {\em pulled tight} if they satisfy one of:
\begin{itemize}
	\item $\gamma = \gamma'$, or
	\item $\gamma$ and $\gamma'$ intersect transversely, and if $t_1, t_2, t_1', t_2' \in [0,1]$ satisfy the property that $\gamma([t_1, t_2]) \cup  \gamma'([t_1',t_2'])$ bounds an imbedded disk $A \subset D_n$, then $A \cap \{p_1, \ldots, p_n\} \neq \emptyset$ (i.e., $\gamma$ and $\gamma'$ are transverse and form no empty bigons).
\end{itemize}
\end{definition}

Note that if $\gamma, \gamma'$ are admissible arcs, there exist admissible arcs $\delta, \delta'$ isotopic to $\gamma,\gamma'$, resp., such that $\delta, \delta'$ are pulled tight (cf. \cite[Chp. 10]{OrderingBraids}). 

\begin{definition}
Let $\gamma, \gamma'$ be two admissible arcs in $D_n$ from $\gamma(0) = \gamma'(0)$. We say $\gamma$ is {\em right} of $\gamma'$ if, when pulled tight via isotopy, either $\gamma = \gamma'$ or the orientation induced by the tangent vectors $\frac{d\gamma}{dt}\vline_{t=0}, \frac{d\gamma'}{dt}\vline_{t=0}$ agrees with the standard orientation on $D \subset \C$.
\end{definition}

\begin{definition} \cite{Rudolph_AlgFunc}  \label{defn:QP} A braid $\sigma \in \Braid_n$ is said to be {\em quasipositive} if \[\sigma = \prod_{j=1}^c w_j \sigma_{i_j} w_j^{-1}\] for some choice of braid words $w_1, \ldots, w_c$.
\end{definition}

\begin{remark} A quasipositive braid is one that is expressible as a product of positive half-twists about regular neighborhoods of {\em arbitrary} admissible arcs in $D_n$ connecting pairs of points in $\Delta$.
\end{remark}

\begin{definition} \label{defn:RV} \cite{HKM_RV, BaldGrig} A braid $\sigma$ is said to be {\em right-veering} if, for each {\em admissible} arc $\gamma$ from $-1 \in \C$ to $\Delta$, $(\gamma)\sigma$ is right of $\gamma$ when pulled tight.
\end{definition}

\begin{remark} Every quasipositive (QP) braid is right-veering (RV). This can be seen directly, but also follows immediately from \cite[Cor. 1]{Plam} and \cite[Prop. 3.10]{BaldGrig}. It is immediate from the definitions that the set of quasipositive braids forms a monoid in the braid group, as does the set of right-veering braids (cf. \cite{EtnyreVHM}). Moreover, Orevkov has proved that membership in the set of quasipositive braids is a transverse invariant \cite{OrevkovQPTrans}. On the other hand, membership in the set of right-veering braids is {\em not} a transverse invariant. Indeed, it is shown in \cite{Plam} (see also \cite{HKM_RV}) that any braid is conjugate to one with a right-veering positive stabilization.
\end{remark}

The following lemma, relating the $j_t$ gradings of canonical Lee homology classes associated to opposite orientations of the same link, will be important for our applications.

\begin{lemma} \label{lem:opporient} Let $(L,o) \subset A \times I$ be an oriented annular link. Then for all $t \in [0,2]$, \[d_t(L,o) = d_t(L,\overline{o})= \mbox{min}\{\mbox{gr}_{j_t}[{\bf x}] \,\,| \,\, [{\bf x}] \in \mbox{Span}_\F\{[\mathfrak{s}_o], [\mathfrak{s}_{\overline{o}}]\}\}.\] 
\end{lemma}

\begin{proof} As proved in \cite[Lem. 3.5, Cor. 3.6]{Rasmussen_Slice}, the Lee complex of a link splits into two subcomplexes according to the remainder of the $j$--grading mod $4$, and the classes $\mathfrak{s}_+ = \mathfrak{s}_o - \mathfrak{s}_{\overline{o}}$ and $\mathfrak{s}_- = \mathfrak{s}_o + \mathfrak{s}_{\overline{o}}$ lie in different subcomplexes. An argument analogous to the one Rasmussen uses to prove that $s_{min} = s([\mathfrak{s}_o]) = s([\mathfrak{s}_{\overline{o}}])$ will then tell us that for each $t \in [0,2]$, one of $\mathfrak{s}_\pm$ has {\em minimum} $j_t$ grading among all classes in $\mbox{Span}_\F\{[\mathfrak{s}_o], [\mathfrak{s}_{\overline{o}}]\}$ (the other has {\em maximum} $j_t$ grading), and $\mbox{gr}_{j_t}[\mathfrak{s}_o]$, $\mbox{gr}_{j_t}[\mathfrak{s}_{\overline{o}}]$ both agree with this minimum.

Explicitly, if $L \subseteq A \times I$ is an annular link of $\ell$ components, recall that the {\em wrapping number}, $\omega$, is defined to be the minimum geometric intersection number of $L$ with a meridional disk of $A \times I$, where this minimum is taken among all isotopy class representatives of $L$.  

We can now define two subcomplexes of the annular Khovanov-Lee complex, which coincide with those Rasmussen defines.

First, let \begin{eqnarray*}
	L^e &:=& \{(a,b) \in \Z^2 \,\,|\,\, a \equiv \ell \mod 4\,\,\, \mbox{ and } \,\,\, b \equiv \ell + 2\omega \mod 4\},\\
	L^o &:=& (2,2) + L^e.
\end{eqnarray*}

Note that $L^e \cap L^o = \emptyset$.

Now plot the annular Khovanov-Lee complex on the $\Z^2$ lattice according to $(j_0,j_2)$ gradings of its distinguished basis (see Remark \ref{rmk:lattice}) and note that since 
\begin{itemize}
	\item the $j_0$ grading of a distinguished basis element agrees mod $2$ with $\ell$ (cf. \cite{Lee}, \cite{Rasmussen_Slice}) and
	\item the $j_0 - j_2 = 2k$ grading of a distinguished basis element agrees mod $4$ with $2n$ (cf. \cite{LRoberts}),
\end{itemize}
the Khovanov-Lee complex is supported in $L^e \cup L^o$. Moreover, the fact (Lemma \ref{lem:Leediffgr}) that the Lee differential preserves both the $j_0$ and $j_2$ gradings mod $4$ tells us that the annular Khovanov-Lee complex splits as a direct sum of the two subcomplexes $\mathcal{C}^e$ and $\mathcal{C}^o$ supported on $L^e$ and $L^o$, respectively.

Indeed, by forgetting the $j_2$ grading, we see that $\mathcal{C}^e$ and $\mathcal{C}^o$ coincide with the subcomplexes Rasmussen defines in \cite[Lemma 3.5]{Rasmussen_Slice}. In particular, $\mathfrak{s}_+$ is contained in one of the two subcomplexes, and $\mathfrak{s}_-$ is contained in the other.

But now notice that the fact that $L^e$ and $L^o$ are disjoint ensures that their projections to any line of irrational slope must also be disjoint. This tells us that for any $t \not\in \mathbb{Q}$, \[\mbox{gr}_{j_t}[\mathfrak{s}_+] \neq \mbox{gr}_{j_t}[\mathfrak{s}_-].\] This in turn implies that for any $t \not\in \mathbb{Q}$, \[\mbox{gr}_{j_t}[\mathfrak{s}_o] = \mbox{gr}_{j_t}[\mathfrak{s}_{\overline{o}}] = \mbox{min}\{\mbox{gr}_{j_t}[\mathfrak{s}_+], \mbox{gr}_{j_t}[\mathfrak{s}_-]\}.\]

By continuity of $d_t$ with respect to $t$, the result follows for all rational $t$ as well. That is, for all $t \in [0,2]$, \[d_t(L,o) = d_t(L,\overline{o}) = \mbox{min}\{\mbox{gr}_{j_t}[{\bf x}] \,\,|\,\, [{\bf x}] \in \mbox{Span}_\F\{[\mathfrak{s}_o],[\mathfrak{s}_{\overline{o}}]\},\] as desired.

\end{proof}

\begin{lemma} \label{lem:sliceBenn}
Let $\sigma \in \Braid_n$ have writhe $w$. Then \[-n|1-t|+w \leq d_t(\widehat{\sigma})\] for all $t \in [0,2]$.
\end{lemma}

\begin{proof}
When $t \in [0,1]$, we calculate $\mbox{gr}_{j_t}(\mathfrak{s}_{o_\uparrow}(\widehat{\sigma})) = (-n+w) - t(-n)$, so \[\mbox{gr}_{j_t}(\mathfrak{s}_{o_\uparrow}(\widehat{\sigma})) = -n(1-t) + w \leq \mbox{gr}_{j_t}([\mathfrak{s}_{o_\uparrow}(\widehat{\sigma})]) = d_t(\widehat{\sigma}),\] as desired. The extension of the bound to $t \in [1,2]$ follows from Theorem \ref{thm:main}, part (2).
\end{proof}

\begin{remark} Lemma \ref{lem:sliceBenn} recovers the Plamenevskaya-Shumakovitch ``$s$--Bennequin inequality" (\cite[Prop. 4]{Plam}, \cite[Lemma 1.C.]{Shum}) when $t=0$. That is:
\[\mbox{sl}(\widehat{\sigma}) \leq s(\widehat{\sigma}) - 1\] where $\mbox{sl}(\widehat{\sigma}) = -n+w$ is the self-linking number of the transverse link represented by $\widehat{\sigma}$. See Section \ref{sec:transverse}. 
\end{remark}

The ``$s$--Bennequin bound" is sharp for quasipositive braid closures (implicit in \cite[Rmk. 2]{Plam}, explicit in \cite[Prop. 1.F.]{Shum}). Indeed, this argument can be extended to show:

\begin{theorem} \label{thm:QP}
If $\sigma$ is a quasipositive braid of index $n$ and writhe $w \geq 0$, we have \[d_t(\widehat{\sigma}) = -n|1-t| + w\] for all $t \in [0,2]$.
\end{theorem}

\begin{proof} Lemma \ref{lem:sliceBenn} gives us the lower bound \[-n|1-t| + w \leq d_t(\widehat{\sigma}).\] To obtain the upper bound, note that if $\widehat{\sigma}$ is quasipositive, there is an oriented (annular) cobordism $F$ from $\widehat{\sigma}$ to $\widehat{\Id}_n$ obtained by performing an orientable saddle cobordism near each quasipositive generator of $\sigma$ as in \cite[Fig. 7]{Plam}, and each component of this cobordism has a boundary on $\widehat{\sigma}$. An easy computation using the crossingless diagram for $\widehat{\Id}_n$ tells us that $d_t(\widehat{\Id}_n) = -n|1-t|$, so part (5) of Theorem \ref{thm:main} tells us: \[d_t(\widehat{\sigma}) \leq -n|1-t| + w,\] as desired. 
\end{proof}

Theorem \ref{thm:QP} above provides an obstruction to a braid conjugacy class being quasipositive. Unfortunately, Theorem \ref{thm:d1=w} tells us that this obstruction is no more sensitive than the one coming from the sharpness of the $s$--Bennequin bound.

\begin{theorem} \label{thm:d1=w} Let $\sigma \in \Braid_n$ have writhe $w$. Then $d_1(\widehat{\sigma}) = w$.
\end{theorem}

\begin{proof}
Let $\mathcal{D}$ be a diagram of an annular braid closure $\widehat{\sigma} \subseteq A \times I$ and let $\mathcal{C}$ denote the graded vector space underlying the Lee complex. Following \cite{Lee} and \cite{Rasmussen_Slice}, we note that the vector space, $\mathcal{C}$, is generated by resolutions of $\mathcal{D}$ whose circles are labeled by ${\bf a}$ or ${\bf b}$, where ${\bf a} = v_- + v_+$ and ${\bf b} = v_- - v_+$. Indeed the set of ${\bf a}/{\bf b}$ markings of resolutions of $\mathcal{D}$ forms a basis for $\mathcal{C}$.\footnote{This is {\em not} a filtered graded basis for the $(\Z \oplus \Z)$--filtered complex $(\mathcal{C},\partial^{Lee})$, but it is a basis.} For the purposes of this proof, let us denote the set of these generators by $S$. We will consider a partition of $S$ into three subsets:

\begin{itemize}   
	\item $S_1 = \{\mathfrak{s}_{o_\uparrow}\}$
	\item $S_2 = \{{\bf x} \in (S \setminus S_1) \,\,|\,\, {\bf x} \mbox{ is a labeling of the braid-like resolution of } \mathcal{D}\}$
	\item $S_3 = S \setminus (S_1 \cup S_2)$
\end{itemize}

Recall that the {\em braid-like resolution} of $\mathcal{D}$ is the oriented resolution for the braid-like orientation, $o_\uparrow$. Corresponding to this partition of S, there is a direct sum decomposition of $\mathcal{C}$ into subspaces: $\mathcal{C} = V_1 \oplus V_2 \oplus V_3$ with $V_i = \mbox{Span}(S_i)$.

Let \begin{eqnarray*}
	p &:& \mathcal{C} \rightarrow V_1 \oplus V_2\\
   	q &:& V_1 \oplus V_2 \rightarrow V_1
\end{eqnarray*} denote the projection maps. Note that $p$ and $q$ satisfy the following properties.
\begin{enumerate}
	\item With respect to the $j_1$--grading on $\mathcal{C}$, $p$ is a grading-preserving map of graded vector spaces. 
	\item $q \circ p$ is a chain map.
	\item $(q\circ p)(\mathfrak{s}_{o_\uparrow}) = \mathfrak{s}_{o_\uparrow}$.
\end{enumerate}

We can now prove the following claims.

\begin{claim} \label{claim:projnonzero} If $z \in \mathcal{C}$ is a cycle satisfying $[z] = [\mathfrak{s}_{o_\uparrow}]$, then $p(z) \neq 0 \in \mathcal{C}$.
\end{claim}

\begin{proof} 
Let $z$ be a cycle representing $[\mathfrak{s}_{o_\uparrow}]$. Then we can write \[z  =  \mathfrak{s}_{o_\uparrow} + \partial^{Lee}(x)\] for some $x \in \mathcal{C}$, and hence \[(q\circ p)(z)  =  (q\circ p)(\mathfrak{s}_{o_\uparrow}) + (q\circ p)(\partial^{Lee} x)  =  \mathfrak{s}_{o_\uparrow} + \partial^{Lee}(q\circ p)(x)\] by properties 2 and 3. Thus $[(q\circ p)(z)]  =  [\mathfrak{s}_{o_\uparrow}]$, and since $[\mathfrak{s}_{o_\uparrow}]$ is nonzero, this shows that $(q \circ p)(z)$ and hence $p(z)$ is nonzero.
\end{proof}

\begin{claim} $d_1(\widehat{\sigma}) \leq w$.
\end{claim}

\begin{proof} Let $z$ be a representative of $[\mathfrak{s}_{o_\uparrow}]$ for which $\mbox{gr}_{j_1}(z) = \mbox{gr}_{j_1}[\mathfrak{s}_{\uparrow}]$. Then $p(z)$ is nonzero by Claim \ref{claim:projnonzero}, and since all elements in $V_1 \oplus V_2$ have $\mbox{gr}(j_1) = w$, we have $\mbox{gr}_{j_1}(p(z)) = w$. Since $p$ is graded as a map of $j_1$--graded vector spaces, this implies\footnote{From the point of view of the ``racing team" analogy from the end of Section \ref{sec:algprelim}, this implication is clear: If the speed of a team is equal to the speed of its slowest member, then you will never reduce the speed of the team by removing a member. (Unless you remove the last member of the team.)} \[d_1(\widehat{\sigma})  =  \mbox{gr}_{j_1}(z) \leq \mbox{gr}_{j_1}(p(z))  =  w.\]  
\end{proof}

On the other hand, \[d_1(\widehat{\sigma}) = \mbox{gr}_{j_1}[s_{o_\uparrow}] \geq \mbox{gr}_{j_1}(s_{o_\uparrow}) = w,\] so \[d_1(\widehat{\sigma}) =  w,\] as desired.

\end{proof}

\begin{corollary} \label{cor:sminus1=sl} If $\sigma \in \Braid_n$ has writhe $w$, then \[d_t(\widehat{\sigma})= -n|1-t| + w \,\,\mbox{ iff } \,\, \mbox{sl}(\widehat{\sigma}) = s(\widehat{\sigma}) - 1.\]
\end{corollary}

\begin{proof} Recalling that $\mbox{sl}(\widehat{\sigma}) = -n + w$, the forward implication follows from setting $t=0$ and applying part (3) of Theorem \ref{thm:main}. The reverse implication follows from Theorem \ref{thm:d1=w} and part (4) of Theorem \ref{thm:main}. In particular, $d_t(\widehat{\sigma})$ is piecewise linear, and $m_t(\widehat{\sigma}) \leq n$ for all $t \in [0,2]$. But if $d_0(\widehat{\sigma}) = -n+w,$ and $d_1(\widehat{\sigma}) = w$, then $m_t(\widehat{\sigma}) = n$ for all $t \in [0,1]$. Part (2) of Theorem \ref{thm:main} completes the argument.
\end{proof}

We also have a {\em sufficient} condition for a braid conjugacy class to be right-veering, given by the following.

\begin{theorem} \label{thm:RV} Let $\sigma \in \Braid_n$. If $m_{t_0}(\widehat{\sigma}) = n$ for some $t_0 \in [0,1)$, then $\sigma$ is right-veering.
\end{theorem}

\begin{proof} We will prove the contrapositive: that if $\sigma$ is not right-veering, then $m_{t_0}(\widehat{\sigma}) < n$ for all $t_0 \in [0,1)$.

If $\sigma$ is not right-veering, then \cite[Cor. 16]{HS} tells us that $\kappa(\widehat{\sigma}) = 2$. Recall from \cite[Defn. 1]{HS} that 
\[\kappa(\widehat{\sigma}) := n + \min\{c \,\,|\,\, \psi(\widehat{\sigma}) = 0 \in H_*(\cF_c)\},\] where here $\psi(\widehat{\sigma}) \in \mbox{Kh}(\widehat{\sigma})$ is Plamenevskaya's invariant, and $\cF_c$ is the subcomplex of the {\em Khovanov} complex generated by distinguished basis elements whose $k$--grading is at most $c$. That is, letting $\mathcal{C}^{Kh}$ denote the graded vector space underlying the Khovanov complex associated to (an annular diagram for) $\widehat{\sigma} \subseteq A \times I$:
\[\cF_c := \mbox{Span}_\F\{{\bf x} \in \mathcal{C}^{Kh} \,\,|\,\, \mbox{gr}_{k}({\bf x})  \leq c\}.\] 

Now consider the {\em Lee} complex $\mathcal{C}^{Lee}(\widehat{\sigma})$, plotted on the $(j_0,j_2)$ lattice. It is helpful to note that in the plane spanned by this lattice:
\begin{itemize}
	\item vertical lines have constant $j$ grading,
	\item slope $1$ lines (i.e., those along which $\mbox{gr}_{j_0} - \mbox{gr}_{j_2}$ is constant) have constant $k$ grading, and
	\item slope $-1$ lines (i.e., those along which $\mbox{gr}_{j_0} + \mbox{gr}_{j_2}$ is constant) have constant $j_1$ grading.
\end{itemize}	

Plamenevskaya's cycle ${\bf v}_-$ has $j$ grading $-n+w$ and $k$ grading $-n$, hence is supported in lattice point $(-n+w, n+w)$. Moreover, ${\bf v}_-$ is the {\em unique} distinguished basis element whose $k$ grading is $-n$, so the lattice point containing ${\bf v}_-$ is the {\em only} lattice point $(a,b) \in \Z^2$ containing a distinguished basis element and satisfying $\mbox{gr}_k(a,b) = -\frac{a-b}{2} = n$.

Now suppose (aiming for a contradiction) that $m_{t_0} = n$ for some $t_0 \in [0,1)$. Then by the above observation there exists some cycle $\xi \in \mathcal{C}^{Lee}$ satisfying:
\begin{itemize}
	\item $[\xi] = [\mathfrak{s}_{o_\uparrow}] \in H_*(\mathcal{C}^{Lee})$,
	\item $\xi = c{\bf v}_- + \xi'$ for some $c \neq 0$, with $\xi'$ supported in $\Z^2 \setminus \{(-n+w,n+w)\}$, and
	\item $\mbox{gr}_{j_{t_0}}(\xi) = \mbox{gr}_{j_{t_0}}([\mathfrak{s}_{o_\uparrow}])$.
\end{itemize}

But we also know that $\kappa(\widehat{\sigma}) = 2$; i.e., there exists a chain $\theta$ in $\mathcal{C}^{Kh} = \mathcal{C}^{Lee}$, the graded vector space underlying the annular Lee complex, satisfying:
\begin{itemize}
	\item $(\partial_0 + \partial_-)(\theta) = {\bf v}_-$ and
	\item $\mbox{gr}_k(\theta) \leq -n+2$
\end{itemize}

Moreover, since 
\begin{itemize}
	\item $\mbox{deg}_{k}(\partial_0) = 0$, 
	\item $\mbox{deg}_{(j,k)}(\partial_-) = (0,-2)$, and 
	\item $\mathcal{F}_{-n}$ is $1$--dimensional, generated by ${\bf v}_-$, 
\end{itemize}
we know that $\theta$ is supported in $(j_0,j_2)$ lattice point $(-n+w,n+w - 4)$, and $\partial_0(\theta) = 0$. 

Recall (see the proof of part (4) of Theorem \ref{thm:main}) that $m_{t_0}(\widehat{\sigma}) = n$ implies that \[\min_{(a,b) \in \mbox{supp}(\xi)}\{\mbox{gr}_{j_{t_0}}(a,b)\} = \mbox{gr}_{j_{t_0}}({\bf v}_-),\] and we calculate that \[\mbox{gr}_{j_{t_0}}({\bf v}_-) = (-n+w)\left(1 - \frac{t_0}{2}\right) + (n+w)\left(\frac{t_0}{2}\right).\] There are now two possibilities:
\begin{enumerate}
	\item There exists $(a',b') \in \mbox{supp}(\xi')$ for which \[\mbox{gr}_{j_{t_0}}(a',b') = \mbox{gr}_{j_{t_0}}({\bf v}_-) = (-n+w)\left(1 - \frac{t_0}{2}\right) + (n+w)\left(\frac{t_0}{2}\right).\]
	\item There does not exist such an $(a',b') \in \mbox{supp}(\xi')$, which implies that \[a'\left(1 - \frac{t_0}{2}\right) + b'\left(\frac{t_0}{2}\right) > (-n+w)\left(1 - \frac{t_0}{2}\right) + (n+w)\left(\frac{t_0}{2}\right)\] for all $(a',b') \in \mbox{supp}(\xi')$.
\end{enumerate}

In Case (1), there exists $\delta >0$ for which $\mbox{gr}_{j_t}(a',b') < \mbox{gr}_{j_t}(-n+w,n+w)$ for all $t \in (t_0, t_0+\delta)$, which tells us that $\mbox{gr}_{j_t}(\xi) \neq \mbox{gr}_{j_t}({\bf v}_-)$ for $t \in (t_0, t_0 + \delta)$, contradicting the assumption that \[m_{t_0}(\widehat{\sigma})  = m_{t_0}(-n+w,n+w) = n.\]

\begin{figure}
\centering
\includegraphics[width=6cm]{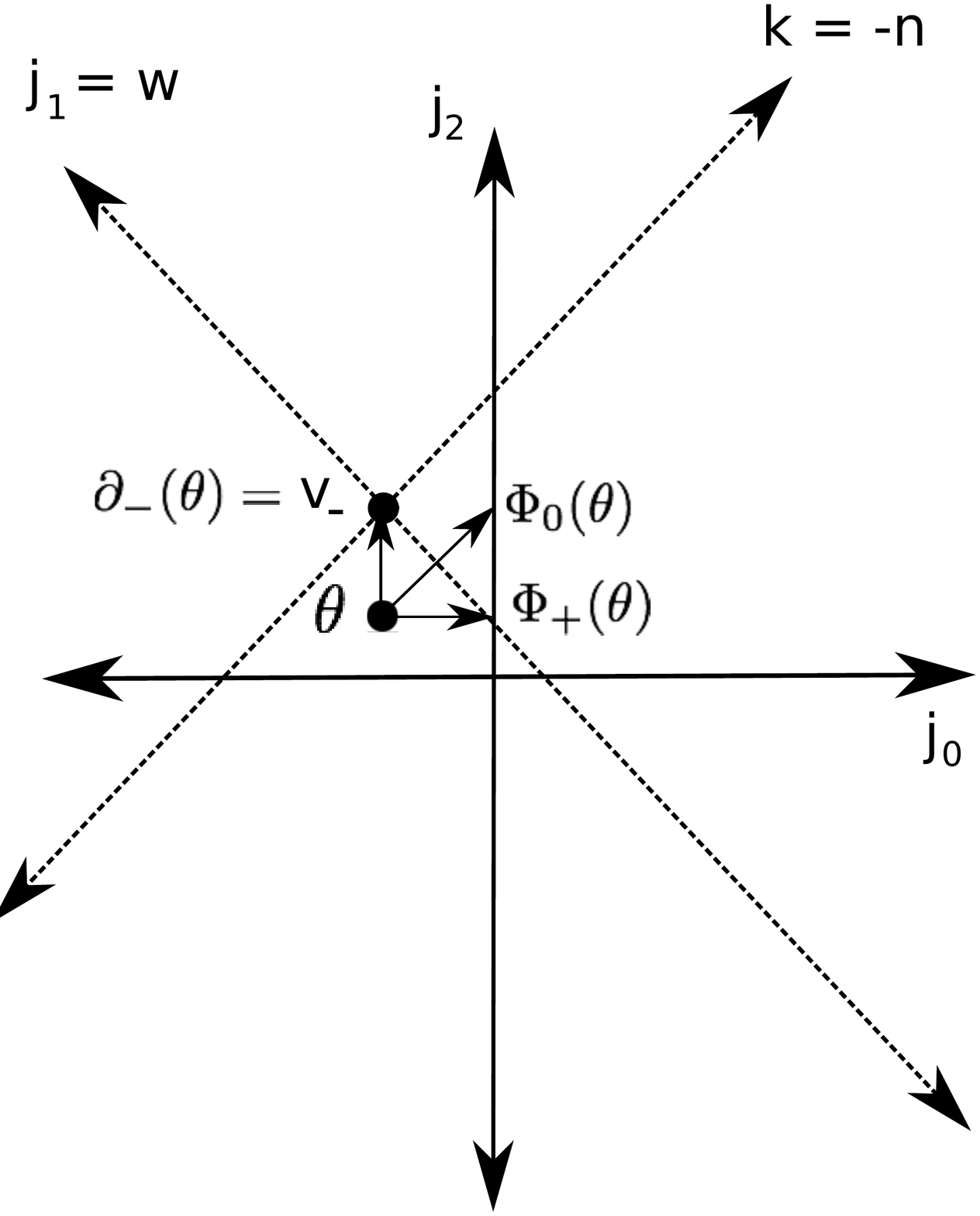}
\caption{If $\sigma \in \Braid_n$ has writhe $w$ and $\kappa(\widehat{\sigma}) = 2$, there exists $\theta \in \mathcal{C}^{Lee}(\widehat{\sigma})$ with $\mbox{gr}_{j_0}(\theta) = \mbox{gr}_{j_0}({\bf v}_-) = -n+w$ and $\mbox{gr}_k(\theta) = -n+2$ with $\partial(\theta) = \partial_-(\theta) = {\bf v}_-$. But this implies that for each $t \in [0,1)$, the $j_t$--grading of a class $\xi$ representing $[\mathfrak{s}_{o_\uparrow}]$ with nontrivial projection to $\mbox{Span}\{{\bf v}_-\}$ can be improved by subtracting an appropriate multiple of $(\partial + \Phi)(\theta)$. In the language of the ``racing team" analogy of Section \ref{sec:algprelim}, the team members $\Phi_0(\theta)$ and $\Phi_+(\theta)$ are ``faster" than $\partial_-(\theta)$, regardless of which $j_t$ is the judge. So a team can always be improved by replacing $\partial_-(\theta) = {\bf v}_-$ by $\Phi_0(\theta)$ and $\Phi_+(\theta)$.}
\label{fig:Kappa2}
\end{figure}

In Case (2), consider $\xi'' := \xi - (\partial + \Phi)(c\theta)$, which satisfies $[\xi''] = [\xi] = [\mathfrak{s}_{o_\uparrow}] \in H_*(\mathcal{C}^{Lee},\partial^{Lee})$. Recalling the degrees of $\partial_-, \Phi_0,$ and $\Phi_+$ (Lemma \ref{lem:Leediffgr}), let \[S := \mbox{supp}(\xi'') \subseteq \mbox{supp}(\xi') \cup \{(-n+w+4,n+w - 4), (-n+w + 4, n+w)\}.\] One now easily verifies (cf. Figure \ref{fig:Kappa2}) that \[\mbox{gr}_{j_{t_0}}(\xi'') = \min_{(a,b) \in S}\left\{a\left(1 - \frac{t_0}{2}\right) + b\left(\frac{t_0}{2}\right)\right\}\] is strictly greater than $\mbox{gr}_{j_{t_0}}({\bf v}_-)$, contradicting the assumption that $\mbox{gr}_{j_{t_0}}(\xi) = \mbox{gr}_{j_{t_0}}[\mathfrak{s}_{o_\uparrow}]$.

\end{proof}

On the other hand, the possible subsets of $[0,1)$ upon which $m_t(\widehat{\sigma}) = n$ is severely constrained by Theorem \ref{thm:d1=w}:

\begin{proposition} \label{prop:maxconstrain} Let $\sigma \in \Braid_n$ have writhe $w$. If $m_{t_0}(\widehat{\sigma}) =n$ for some $t_0 \in [0,1)$, then $m_t (\widehat{\sigma}) = n$ for all $t \in [t_0, 1)$.
\end{proposition}

\begin{proof} As in the proof of Theorem \ref{thm:RV}, we observe that there is a unique distinguished basis vector, ${\bf v}_-$, with $k$--grading $-n$. In view of Remark \ref{rmk:kgrading}, this implies that if $m_{t_0}(\widehat{\sigma}) = n$, then \[j_{t_0}([\mathfrak{s}_{o_\uparrow}(\widehat{\sigma})]) = j_{t_0}({\bf v}_-) = (-n+w) + nt_0.\] But Theorem \ref{thm:d1=w}, combined with the fact (part (4) of Theorem \ref{thm:main}) that $m_t(\widehat{\sigma})\leq n$ for all $t \in [0,2]$ then implies that $m_t(\widehat{\sigma}) = n$ for all $t \in [t_0, 1)$, as desired.
\end{proof}

The annular Rasmussen invariants are additive under horizontal composition:

\begin{proposition} \label{prop:additivity} Let $(L,o), (L',o') \subset A \times I$, and let $(L,o) \amalg (L',o') \subset A \times I$ denote their {\em annular composition}, as in Figure \ref{fig:horizstack}. Then for all $t \in [0,2]$, \[d_t((L,o) \amalg (L',o')) = d_t(L,o) + d_t(L',o').\]
\end{proposition}

\begin{figure}
\centering
\includegraphics[width=6cm]{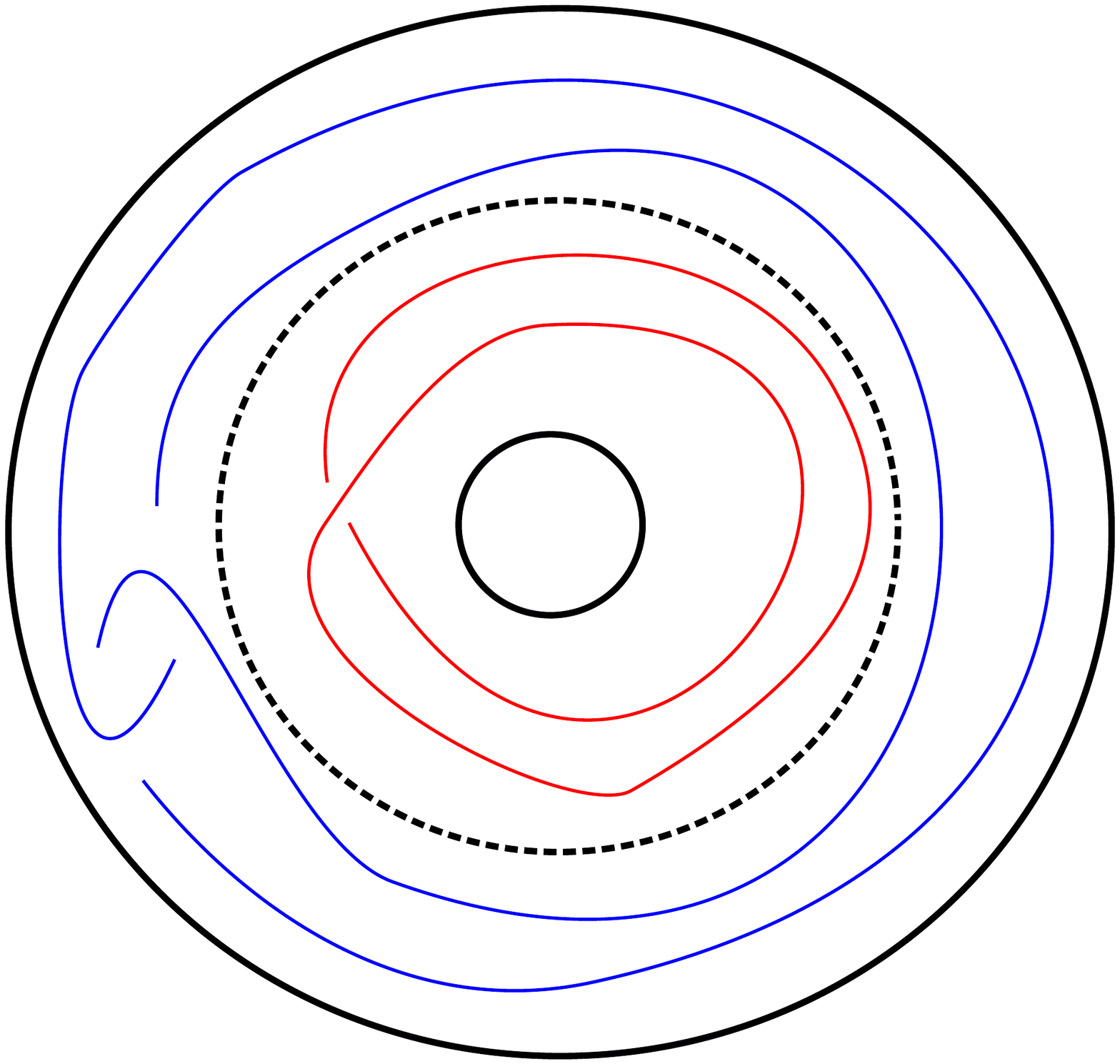}
\caption{A diagram of the annular composition of two links $L, L' \subset A \times I$ is formed by identifying the inner boundary of the annulus containing (a diagram of) $L$ (in blue) with the outer boundary of the annulus containing (a diagram of) $L'$ (in red).}
\label{fig:horizstack}
\end{figure}

\begin{proof} As $\Z \oplus \Z$--filtered complexes, \[\mathcal{C}^{Lee}(L \amalg L') = \mathcal{C}^{Lee}(L) \otimes \mathcal{C}^{Lee}(L').\] Moreover, $\mathfrak{s}_{o \amalg o'}$ is $\mathfrak{s}_{o} \otimes \mathfrak{s}_{o'}$ (resp., $\mathfrak{s}_{o} \otimes \mathfrak{s}_{\overline{o}'}$) when the wrapping number, $\omega$, of $L$ is even (resp., odd).  Appealing to Lemma \ref{lem:opporient}, this implies (cf. \cite[Prop. 1.8]{OSSUpsilon} and \cite[Thm. 7.2]{LivingstonUpsilon}) that for each $t \in [0,2]$, $j_t[\mathfrak{s}_{o \amalg o'}] = j_t[\mathfrak{s}_o] + j_t[\mathfrak{s}_{o'}]$, as desired. 
\end{proof}

\begin{corollary} \label{cor:additivity} Let $(L,o), (L',o') \subset A \times I$. Then for all $t \in [0,2]$, \[m_t((L,o) \amalg (L',o')) = m_t(L,o) + m_t(L',o').\]
\end{corollary}

\begin{proposition} \label{prop:stabbound} Let $\sigma \in \Braid_n$, and suppose $\sigma^{\pm} \in \Braid_{n+1}$ is obtained from $\sigma$ by either a positive or negative Markov stabilization. Then for all $t \in [0,1],$ \[d_t(\widehat{\sigma}) - t \leq d_t(\widehat{\sigma}^{\pm}) \leq d_t(\widehat{\sigma}) + t.\]
\end{proposition}

\begin{proof} Consider the obvious oriented annular cobordism from $\widehat{\sigma}^{\pm}$ to $\widehat{\sigma} \amalg \widehat{\Id}_1$ (the horizontal composition of $\widehat{\sigma}$ with the trivial $1$--braid closure) with a single odd-index critical point that resolves the extra $\pm$ crossing. By \cite[Prop. 4.1]{Rasmussen_Slice}, we see that the associated chain map on the Lee complex sends $\mathfrak{s}_{o_\uparrow}(\widehat{\sigma}^\pm)$ to $\mathfrak{s}_{o_\uparrow}(\widehat{\sigma} \amalg \widehat{\Id}_1)$. By Proposition \ref{prop:cobdegree}, this map is filtered of degree $-1$ for all $t \in [0,1]$. Therefore, \[d_t(\widehat{\sigma}^+) - 1 \leq d_t(\widehat{\sigma} \amalg \widehat{\Id}_1).\] Proposition \ref{prop:additivity} and Theorem \ref{thm:QP} then tell us $d_t(\widehat{\sigma} \amalg \widehat{\Id}_1) = d_t(\widehat{\sigma}) + (-1 + t)$, which gives us one of the two desired inequalities: \[d_t(\widehat{\sigma}^\pm) \leq d_t(\widehat{\sigma}) + t.\]

To obtain the other inequality, consider the quasi-isomorphism \[\rho_1': \mathcal{C}^{Lee}(\widehat{\sigma}) \rightarrow \mathcal{C}^{Lee}(\widehat{\sigma}^+)\] appearing in \cite[Sec. 6]{Rasmussen_Slice}. We have $\rho_1'(\mathfrak{s}_{o_\uparrow}(\widehat{\sigma})) = \mathfrak{s}_{o_\uparrow}(\widehat{\sigma}^+)$. Moreover, we compute that the filtered degree of $\rho_1'$ with respect to the $j_t$--grading is $-t$ (as usual, there are higher order terms: they are of degree $4-t, 4-3t,$ and $t$). This tells us that \[d_t(\widehat{\sigma}) - t \leq d_t(\widehat{\sigma}^+).\] A similar argument using the R1 map associated to the negative stabilization 
tells us \[d_t(\widehat{\sigma}) - t \leq d_t(\widehat{\sigma}^-)\] as well.

\end{proof}

\section{Transverse invariants and annular Khovanov-Lee homology} \label{sec:transverse}

In this section we study what annular Khovanov-Lee theory can tell us about transverse isotopy classes of transverse links with respect to the standard tight contact structure $\xi_{st}$ on $S^3$ (cf. \cite{EtnyreLegTransSurvey} for a survey of this topic). We have transverse versions of the classical Alexander and Markov theorems that allow us to study transverse links via closed braids:

\begin{theorem} \label{thm:TransAlex} \cite{Benn} Every transverse link $L \subset (S^3, \xi_{st})$ is transversely isotopic to a closed braid.
\end{theorem}

\begin{theorem} \label{thm:TransMark} \cite{OrShev,Wrinkle} Two closed braids are transversely isotopic iff they are related by a finite sequence of closed braid isotopies (braid isotopies and conjugations) and {\em positive} stabilizations and destabilizations.
\end{theorem}

Recall that the {\em self-linking number} is a classical invariant of a transverse link $T \subset S^3$ obtained by choosing a trivialization given by a nonzero vector field of $\xi$ over $\Sigma$, a Seifert surface bounded by $T$, and computing the linking number with $T$ of the associated push-off,  $T'$ (cf. \cite[Sec. 2.6.3]{EtnyreLegTransSurvey}). When $T$ is represented by a closed braid $\sigma \in \Braid_n$ with writhe $w$, we have: \[\mbox{sl}(\widehat{\sigma}) = -n + w.\]

The following set was defined in \cite{EtnyreVHM}:

\begin{definition}
$\mathcal{S} := \{\mbox{Braids } \sigma \,\,|\,\, \mbox{sl}(\widehat{\sigma}) = s(\widehat{\sigma}) - 1\}$
\end{definition}

Informally, this is the set of braids for which the s-Bennequin bound is sharp. Since $s(\widehat{\sigma})$ is a link invariant, and $\mbox{sl}(\widehat{\sigma})$ is a transverse link invariant, it follows that membership in $\mathcal{S}$ is an invariant of the transverse isotopy class of $\widehat{\sigma}$. It is proven in \cite{EtnyreVHM} that $\mathcal{S}$ is a monoid. That is, $\Id_n \in \mathcal{S} \cap \Braid_n$, and if $\sigma, \sigma' \in \mathcal{S} \cap \Braid_n$, then so is their product: $\sigma\sigma' \in \mathcal{S} \cap \Braid_n$.

In Definition \ref{defn:maxslope} we define a family of subsets of $\Braid_n$ using the annular Rasmussen invariants. We show directly that each of these subsets is a monoid and that membership in each of the subsets is a transverse invariant. Unfortunately, we also show that each of these subsets agrees with the monoid $\mathcal{S}$, so these transverse invariants all fail to be {\em effective}. Recall that a transverse invariant $f$ is said to be {\em effective} if there exists at least one link $L$ and transverse representatives $T_L, T_L'$ of $L$ satisfying $\mbox{sl}(T_L) = \mbox{sl}(T_L')$ and $f(T_L) \neq f(T_L')$.

\begin{definition} \label{defn:maxslope} Let $t_0 \in [0,1)$. Define \[\mathcal{M}_{t_0} := \{\mbox{Braids } \sigma \,\,|\,\, m_t(\widehat{\sigma}) \mbox{ is maximal for all } t \in [0,t_0).\}\]
\end{definition}

Part (4) of Theorem \ref{thm:main} tells us that $n$ is the maximal possible slope among slopes of $d_t(\widehat{\sigma})$ when $\sigma \in \Braid_n$. Accordingly, we define:

\begin{definition} Let $t_0 \in [0,1)$. Define \[\mathcal{M}^n_{t_0} := \mathcal{M}_{t_0} \cap \Braid_n = \{\sigma \in \Braid_n \,\,|\,\, m_t(\widehat{\sigma}) =n \mbox{ for all } t \in [0,t_0).\}\]
\end{definition}

We have the following:

\begin{theorem} \label{thm:transverse} Let $t_0 \in [0,1)$: 
	\begin{enumerate}
		\item Membership in $\mathcal{M}_{t_0}$ is a transverse invariant
		\item $\mathcal{M}^n_{t_0}$ is a monoid for each $n \in \Z^+$
		\item $\mathcal{M}_{t_0} = \mathcal{S}$
	\end{enumerate}
\end{theorem}

Of course, part (3) of Theorem \ref{thm:transverse}, combined with the results of \cite{EtnyreVHM}, imply parts (1) and (2), but it will be instructive to prove each of the statements directly, as the lemmas involved in the proof may be of independent interest.

\begin{lemma} \label{lem:posband} Let $\sigma \in \Braid_n$ have writhe $w$, and suppose that $\sigma' \in \Braid_n$ is obtained from $\sigma$ by inserting a single positive crossing. Then if $\sigma \in \mathcal{M}_{t_0}$ for some $t_0 \in [0,1)$, then $\sigma' \in \mathcal{M}_{t_0}$.
\end{lemma}

\begin{proof} Since $\sigma \in \mathcal{M}_{t_0}$, we know that for each $t \in [0,t_0)$ we have 
\[d_t(\widehat{\sigma}) = j_t({\bf v}_-(\widehat{\sigma})) = (-n+w) + nt.\] But applying Theorem \ref{thm:main} to the Euler characteristic $-1$ annular cobordism $\widehat{\sigma'} \rightarrow \widehat{\sigma}$ that resolves the single extra positive crossing tells us:
\[d_t(\widehat{\sigma'}) -1 \leq d_t(\widehat{\sigma}) = (-n+w) + nt,\]
and hence: \[d_t(\widehat{\sigma'}) = j_t[\mathfrak{s}_\uparrow(\widehat{\sigma'})] \leq -n+(w+1) + nt.\]

But $j_t(\mathfrak{s}_\uparrow(\widehat{\sigma'})) = -n+(w+1) + nt$. So \[-n+(w+1) + nt \leq d_t(\widehat{\sigma'}),\] hence $m_t(\widehat{\sigma'}) = n$ for all $t \in [0,t_0)$, and $\sigma' \in \mathcal{M}_{t_0}$, as desired. 
\end{proof}

\begin{lemma} \label{lem:posdestab} Let $\sigma \in \Braid_n$, and suppose that $\sigma^+ \in \Braid_{n+1}$ is obtained from $\sigma$ by performing a positive stabilization. Then if $\sigma^+ \in \mathcal{M}_{t_0}$ for some $t_0 \in [0,1)$, then $\sigma \in \mathcal{M}_{t_0}$.
\end{lemma}

\begin{proof} Since $\sigma^+ \in \mathcal{M}_{t_0}$, we know that $d_t(\widehat{\sigma}^+) = (s(\widehat{\sigma}^+) - 1) + (n+1)t$ for all $t \in [0,t_0)$.  Proposition \ref{prop:stabbound} tells us $d_t(\widehat{\sigma}^+) \leq d_t(\widehat{\sigma}) + t$ for all $t \in [0,2]$, so \[(s(\widehat{\sigma}^+) - 1) + (n+1)t \leq d_t(\widehat{\sigma}) +t\] for all $t \in [0,t_0)$. But Theorem \ref{thm:main} tells us that $m_t(\widehat{\sigma}) \leq n$ for all $t \in [0,2]$, and $s(\widehat{\sigma}^+) = s(\widehat{\sigma})$ (since it is an oriented link invariant), so \[d_t(\widehat{\sigma}) \leq (s(\widehat{\sigma}) - 1) + nt,\] hence $\sigma \in \mathcal{M}_{t_0}$, as desired. 
\end{proof}

\begin{lemma} \label{lem:horizstack}
Let $t_0 \in [0,1)$. Let $\sigma \in \Braid_n$ and $\sigma' \in \Braid_{n'}$, and let $(\sigma \amalg \sigma') \in \Braid_{n+n'}$ denote their {\em horizontal composition}. If $\sigma, \sigma' \in \mathcal{M}_{t_0},$ then $\sigma \amalg \sigma' \in \mathcal{M}_{t_0}$. 
\end{lemma}

\begin{proof} This follows immediately from Corollary \ref{cor:additivity}.
\end{proof}

\begin{proof}[Proof of Theorem \ref{thm:transverse}] 
\begin{enumerate}
	\item Theorem \ref{thm:main}, Part (1) tells us that membership in $\mathcal{M}_{t_0}$ is preserved under closed braid isotopies (annular R2 and R3 moves). Lemmas \ref{lem:posband} and \ref{lem:horizstack} imply that membership in $\mathcal{M}_{t_0}$ is preserved under positive stabilization, and Lemma \ref{lem:posdestab} implies that membership in $\mathcal{M}_{t_0}$ is preserved under positive destabilization. Since Theorem \ref{thm:TransMark} tells us two braid closures represent  transversely isotopic links iff they are related by a sequence of closed braid isotopies and positive de/stabilization, membership in $\mathcal{M}_{t_0}$ is a transverse invariant for each $t_0 \in [0,1)$, as desired.
	\item This follows from \cite[Thm. 7.1]{EtnyreVHM}, combined with part (1) above, and Lemmas \ref{lem:posband} and \ref{lem:horizstack}.
	\item We begin by showing that $\mathcal{M}_0 = \mathcal{S}$. 

Let $\sigma \in \Braid_n$ have writhe $w$, and suppose $m_0(\widehat{\sigma}) = n$. This implies that the lattice point $(a,b)$ determining $\mbox{gr}_{j_0}[\mathfrak{s}_{o_\uparrow}]$ satisfies $\mbox{gr}_k(a,b) = -n = \frac{a-b}{2}$.  But because $\mathcal{C}^{Lee}(\widehat{\sigma})$ has a unique irreducible $\sltwo$ subrepresentation of highest weight $n$, and this representation is supported in $j_1$--grading $w$, the only lattice point containing a distinguished filtered graded basis element with this $k$--grading is $(a,b) = (-n+w,n+w)$. So \[\mbox{gr}_{j_0}[\mathfrak{s}_o] = -n+w = \mbox{sl}(\widehat{\sigma}),\] so $\mbox{sl}(\widehat{\sigma}) = \mbox{gr}_{j_0}[\mathfrak{s}_o] = d_0(\widehat{\sigma}) = s(\widehat{\sigma}) - 1$. Hence $\mathcal{M}_0 \subseteq \mathcal{S}$.

To see the reverse inclusion, suppose that $\mbox{sl}(\widehat{\sigma}) = s(\widehat{\sigma}) - 1$. Then the lattice point $(a,b)$ determining $\mbox{gr}_{j_0}[\mathfrak{s}_{o_\uparrow}]$ satisfies $\mbox{gr}_{j_0}(a,b) = a = -n+w$. But note that $\mbox{gr}_{j_0}(\mathfrak{s}_o) = \mbox{gr}_{j_0}({\bf v}_-) = -n+w$ as well. Moreover, since ${\bf v}_-$ has minimal $k$--grading (hence maximal $j_2$--grading) among all distinguished filtered graded basis elements with $j_0$--grading $-n+w$, we have \[\mbox{gr}_{j_t}(\mathfrak{s}_o) = (-n+w)\left(1 - \frac{t}{2}\right) + (n+w)\left(\frac{t}{2}\right) \geq a \left(1 - \frac{t}{2}\right) + b \left(\frac{t}{2}\right)\] for all $t > 0$. This implies that there exists some $\epsilon > 0$ for which $\mbox{gr}_{j_t}[\mathfrak{s}_o] = (-n+w) + nt$ for all $t < \epsilon$. In particular, $m_0(\widehat{\sigma}) = n$. Hence, $\mathcal{S} \subseteq \mathcal{M}_0$, as desired.

The fact that $\mathcal{M}_{t_0} = \mathcal{S}$ for all $t_0 \in [0,1)$ now follows from Proposition \ref{prop:maxconstrain}.
\end{enumerate}
\end{proof}

\begin{remark} \label{rmk:mtprime} There are a number of other closely-related subsets of the braid group we might define in the hopes of constructing effective transverse invariants. For example, if we define: \[\mathcal{M}_{t_0}' := \{\mbox{Braids } \sigma \,\,|\,\,m_t(\widehat{\sigma}) = n \mbox{ for all } t \in [t_0,1)\},\] we might ask whether membership in $\mathcal{M}_{t_0}'$ is a transverse invariant. Indeed, Lemma \ref{lem:horizstack} and the proof of Lemma \ref{lem:posband} carry through verbatim to show that membership in $\mathcal{M}_{t_0}'$ is preserved under positive stabilization, but we are unable to show using the argument in the proof of Lemma \ref{lem:posdestab} that membership in $\mathcal{M}_{t_0}'$ is preserved under positive {\em destabilization}. We therefore also do not know whether $\mathcal{M}_{t_0}'$ is a monoid for $t_0 \neq 0$.
\end{remark}

There is another monoid of interest, described in \cite{BaldPlam} and \cite{EtnyreVHM}. This is the monoid of braids for which Plamenevskaya's invariant is nonzero:

\begin{definition} \[{\Psi} := \{\mbox{Braids } \sigma \,\,|\,\, \psi(\widehat{\sigma}) \neq 0\}\]
\end{definition} 

In the course of the proof of Theorem \ref{thm:RV}, we referenced a related subset of the braid group, studied by Hubbard-Saltz \cite{HS}:

\begin{definition} \[\large{\varkappa} := \{\mbox{Braids } \sigma \,\,|\,\, \kappa(\widehat{\sigma}) > 2\}\]
\end{definition}

By definition, Plamenevskaya's invariant is nonzero iff $\kappa(\widehat{\sigma}) = \infty$. 

It is shown in \cite[Thm. 1.2]{BaldPlam} that $\mathcal{S} \subseteq \Psi$. Combining this and other known results (cf. \cite{Plam}, \cite{BaldPlam}, \cite{EtnyreVHM}, \cite{HS}, \cite{PlamRV}) with the results of this paper, we see:
\[QP \subsetneq (\mathcal{S} = \mathcal{M}_0') \subseteq \ldots \subseteq \mathcal{M}_t' \subseteq \ldots \subseteq \mathcal{M}_{1-\epsilon}' \subseteq \varkappa \subseteq RV\] and \[\mathcal{S} \subseteq \Psi \subseteq \varkappa.\] This leads naturally to the following questions, relevant to the (still open) question of whether Plamenevskaya's transverse invariant is effective (cf. \cite{LipNgSar}). 

\begin{question} \label{question:Psi=S} Is $\Psi \subseteq \mathcal{S}$?
\end{question}

It is shown in \cite[Thm. 1.2]{BaldPlam} that the above inclusion holds among braids representing any $\widetilde{Kh}$--thin knot type.

\begin{remark} If $\Psi \subseteq \mathcal{S}$ for all $n \in \Z^+$, then Plamenevskaya's invariant $\Psi$ is ineffective. This follows since we already know $\mathcal{S} \subseteq \Psi$, so the reverse inclusion would imply that $\psi(\widehat{\sigma}) \neq 0$ iff $\mbox{sl}(\widehat{\sigma}) = s(\widehat{\sigma}) - 1$. In particular, the value of $\psi(\widehat{\sigma})$ would be determined by $\mbox{sl}(\widehat{\sigma})$. On the other hand, it is possible for Plamenevskaya's invariant $\Psi$ to be ineffective but $\Psi \neq \mathcal{S}$.
\end{remark}

\begin{question} \label{question:PsiMax} Is $\Psi \subseteq \mathcal{M}_{t_0}'$ for some $t_0 \in [0,1)$? For all $t_0 \in [0,1)$? How about the reverse inclusion(s)?
\end{question}

\begin{remark} If the answer to Question \ref{question:Psi=S} is yes, then $\Psi \subseteq \mathcal{M}_{t_0}'$ for all $t_0 \in [0,1)$.
\end{remark}

\begin{question} \label{question:kappa} Is $\varkappa \subseteq \mathcal{M}_{t_0}'$ for some $t \in [0,1)$?
\end{question}

\begin{remark} It is shown in \cite[Sec. 4.1]{HS} that $\varkappa \neq \Psi$, so it is not possible for $\varkappa \subseteq (\mathcal{M}_0' = \mathcal{S})$.
\end{remark}

\section{Examples}\label{sec:examples}
In this section we give some example computations of the $d_t$ invariant for some small braids.  Almost all of these computations produce {\em tent-shaped} $d_t$ profiles; that is, in almost all examples we have computed, $d_t$ is a linear function of $t$ with constant slope $m$ on the interval $t \in [0,1)$, and hence a linear function of $t$ with constant slope $-m$ on the interval $t \in [1,2)$.  However, we have not yet computed $d_t$ for braids with more than 11 crossings, and it is distinctly possible that the prevalence of tent-shaped profiles is a feature of considering only small braids.
In the figures below, the numbers on the figures are the values of the $d_t$ invariant at endpoints $t=0,2$ and at the midpoint $t=1$.  The (green) value $d_1$ is the writhe of the braid (Theorem \ref{thm:d1=w}).  

\subsection{$d_t$ detects neither quasipositivity nor the trivial braid.}
The closure of the 3-braid $\s_1^{-5} \s_2 \s_1^{3} \s_2$ is isotopic to the mirror of the the knot $10_{125}$, which is one of the smallest (in crossing number) examples of a non-quasipositive knot whose Rassmussen invariant is equal to the self-linking number plus 1.  The profile $d_t$ of $\s_1^{-5} \s_2 \s_1^{3} \s_2$ is drawn below.
\begin{itemize}
\item $ \s_1^{-5} \s_2 \s_1^{3} \s_2$:
$$
\begin{tikzpicture}
\draw[thick,->] (0,0) -- (4,0) node[anchor=west]{$t$};
\draw[thick] (0,0) -- (-4,0);
\draw[thick,->] (0,0) -- (0,2) node[anchor=south]{$d_t$};
\draw[thick] (0,0) -- (0,-3) ;
\draw[thick,blue] (0,-3) node[anchor=east] {-3} -- (1,0) node[green,anchor=south] {0};
\draw[thick,blue] (1,0) -- (2,-3)node[anchor=west] {-3};
\end{tikzpicture}
$$
\end{itemize}
Note that this example has exactly the same profile as the $d_t$ invariant of the trivial 3-braid.  Since the slope of the above profile is maximal on the interval $[0,1)$, we see that  $\s_1^{-5} \s_2 \s_1^{3} \s_2$ is right-veering (this was previously known--cf. \cite{BaldPlam} and \cite[Prop. 3.2]{PlamRV}).  The mirror of $10_{125}$ is one of only three knots with at most 10 crossings that are not quasipositive but which have a transverse representative whose self-linking number plus 1 equals its Rasmussen invariant. The other two are $10_{130}$ and $10_{141}$; all three were studied in \cite[Sec. 7]{BaldPlam}. Note that by Corollary \ref{cor:sminus1=sl}, each of these examples will also have tent-shaped $d_t$ profiles, with maximal slope on the interval $t \in [0,1)$.

\subsection{$d_t$ does not detect right-veeringness}

The family of right-veering $3$--braids \[\beta_k := (\s_1\s_2)^3 \s_2^{-k} \,\,\,\, (k \geq 1)\] plays a key role in the proof of \cite[Thm. 4.1]{PlamRV}. We are grateful to J. Baldwin for pointing out the following to us:
\begin{enumerate}
	\item When $k = 1, \ldots 4$, $\beta_k$ is quasipositive (this can be seen directly);
	\item When $k = 5$, $\psi(\widehat{\beta}_k) = 0$ (\cite{BaldwinProgram});
	\item Since $\beta_5$ can be obtained from any other braid in this family by adding {\em positive} crossings, it follows from \cite[Thm. 4]{Plam} that $\psi(\widehat{\beta}_k) = 0$ for all $k > 5$ as well.
\end{enumerate}

The $d_t$ profiles of the first few non-quasipositive braids $\beta_k$ are pictured below.

\begin{itemize}
\item $(\s_1\s_2)^3 \s_2^{-5}$ is the top tent (in blue), and $(\s_1\s_2)^3 \s_2^{-6}$ is the middle tent (in red), and $(\s_1\s_2)^3 \s_2^{-7}$ is the bottom tent (in orange):
$$
\begin{tikzpicture}
\draw[thick,->] (0,0) -- (4,0) node[anchor=west]{$t$};
\draw[thick] (0,0) -- (-4,0);
\draw[thick,->] (0,0) -- (0,2) node[anchor=south]{$d_t$};
\draw[thick] (0,0) -- (0,-2) ;
\draw[thick,blue] (0,0) node[anchor=south east] {0} -- (1,1) node[green,anchor=south] {1};
\draw[thick,blue] (1,1) -- (2,0)node[anchor=south west] {0};
\draw[thick,red] (0,-1) node[anchor=south east] {-1} -- (1,0) node[green,anchor=south] {0};
\draw[thick,red] (1,0) -- (2,-1)node[anchor=south west] {-1};
\draw[thick,orange] (0,-2) node[anchor=south east] {-2} -- (1,-1) node[green,anchor=south] {-1};
\draw[thick,orange] (1,-1) -- (2,-2)node[anchor=south west] {-2};
\end{tikzpicture}
$$
\end{itemize}

\subsection{4-strand examples}
The graph below shows the $d_t$ profile of a pair of 4-braids that are not in the monoid \[\mathcal{S} := \{\mbox{Braids } \sigma \,\,|\,\, \mbox{sl}(\widehat{\sigma}) = s(\widehat{\sigma}) - 1\}.\] Note that since $\mathcal{S} = \mathcal{M}_0$ (Theorem \ref{thm:transverse}, part (3)), only braids $\sigma \not\in \mathcal{S}$ can have interesting $d_t$ profiles in the interval $[0,1)$. 

\begin{itemize}

\item $\s_1\s_1\s_2\s_1^{-1}\s_3^{-1}\s_2\s_3^{-1}$ and $\s_1\s_2^{-1}\s_1\s_2^{-1}\s_3\s_2^{-1}\s_3$:

$$
\begin{tikzpicture}
\draw[thick,->] (0,0) -- (4,0) node[anchor=west]{$t$};
\draw[thick] (0,0) -- (-4,0);
\draw[thick,->] (0,0) -- (0,2) node[anchor=south]{$d_t$};
\draw[thick] (0,0) -- (0,-2) ;
\draw[thick,blue] (0,-1) node[anchor=east] {-1} -- (1,1) node[green,anchor=south] {1};
\draw[thick,blue] (1,1) -- (2,-1)node[anchor=west] {-1};
\end{tikzpicture}
$$
\end{itemize}

Despite the fact that the above two braids have the same $d_t$ profile, the Khovanov homology groups of their closures are not isomorphic.  In particular, these two braids are not conjugate in the braid group.

\subsection{\bf Transversely non-isotopic closed braid representatives of the knot $10_{132}$}
It was shown in \cite[Sec. 3.1]{oszt} that the knot $10_{132}$ is transversely non-simple. That is, it has two transversely non-isotopic closed braid  representatives (related by a so-called {\em negative flype}) with the same self-linking number. This knot was studied further by Khandhawit-Ng in \cite{KhNg}, who produced an infinite family of knots with the same property. In \cite[Prop. 3]{HS}, Hubbard-Saltz used their invariant $\kappa$ to distinguish the {\em conjugacy classes} of the transversely non-isotopic closed braid representatives. 

The $d_t$ invariant of the braid $A(0,0) := \sigma_3\sigma_2^{-2}\sigma_3^{2}\sigma_2\sigma_3^{-1}\sigma_1^{-1}\sigma_2\sigma_1^{2}$ (notation from \cite{KhNg} and \cite{HS}) is given below.
 
$$
\begin{tikzpicture}[scale=2.0]
\draw[thick,->] (0,0) -- (3,0) node[anchor=west]{$t$};
\draw[thick] (0,0) -- (-3,0);
\draw[thick,->] (0,0) -- (0,4) node[anchor=south]{$d_t$};
\draw[thick,blue] (0,1)node[red,anchor=east]{1} -- (0.5,2) node[red,anchor=east]{2} ;
\draw[thick,blue] (0.5,2) -- (.66,1.66) node[red,anchor=north]{5/3};
\draw[thick,blue] (0.66,1.66) -- (1,3) node[green,anchor=south] {3};
\draw[thick,blue] (1,3)--(1.33,1.66) node[red,anchor=north]{5/3} ; 
\draw[thick,blue] (1.33,1.66) -- (1.5,2) node[red,anchor=west]{2};
\draw[thick,blue] (1.5,2) -- (2,1) node[red,anchor=west]{1};
\end{tikzpicture}
$$

The profile drawn above is based on the computations of $d_t$ at values $t=\frac{k}{24}$, $k= 1,\dots, 24$, along with the symmetry $d_{1-t} = d_{1+t}$ for $t\in [0,1]$.  Thus it is a priori possible (though we think it unlikely) that there are further points of discontinuity in the profile that we have not computed.  In general, it would be interesting to know a bound -- in terms of combinatorial statistics of a braid $\sigma$ -- on the denominator $n$ such that $d_t(\sigma)$ has points of non-differentiability at $t=\frac{m}{n}$.

The points of non-differentiability in the $t$-interval $(0,2)$ above are $t\in\{1/2,2/3,1,4/3,3/2\}$.  The corresponding $d_t$ values, as indicated on the graph, are 
$$
	d_{1/2}  = d_{3/2} = 2, \ \ d_{2/3} = d_{4/3} = 5/3, \ \ \text{ and } d_1 = 3.
$$

Knotinfo \cite{KnotInfo} tells us that the $4$--ball genus of $10_{132}$ is $1$ and its $s$ invariant is $2$.  Note that since the slope, $m_t$, is maximal on a proper subset of the interval $[0,1)$, the braid must be right-veering but not quasipositive. This was already established by Hubbard-Saltz in \cite{HS}, who computed that $\kappa(A(0,0)) = 4$. Note that $\kappa \neq 2$ implies $A(0,0)$ is right-veering \cite[Cor. 16]{HS}, and $\kappa \neq \infty$ implies that $A(0,0)$ is not quasipositive \cite[Cor. 1]{Plam}.  Note also that the $d_t$ profile for $A(0,0)$ shows that the braid classes $\mathcal{M}'_t$ are not all equal to each other, as
$A(0,0)\in \mathcal{M'}_{\frac{2}{3}}$ but $A(0,0)\notin \mathcal{M'}_0$.

\bibliography{SKhTwoVariable}
\end{document}